\documentclass[bezier,emlines,12pt]{article}
\usepackage[OT1]{fontenc} 
\usepackage{graphicx,xcolor}
\usepackage[pagebackref=true,colorlinks,linkcolor=red,citecolor=blue]{hyperref}

\def\newpic#1{}
\def\mtwo#1#2{\raise .8ex\hbox{
		${#1_{{\displaystyle #2}}}$}}
\usepackage{sansmath}
\usepackage{amsmath} 
\DeclareSymbolFont{extraup}{U}{zavm}{m}{n}
\DeclareMathSymbol{\varheart}{\mathalpha}{extraup}{86}

\DeclareMathSymbol{\vardiamond}{\mathalpha}{extraup}{87}
\DeclareMathSymbol{\varclub}{\mathalpha}{extraup}{88}
\DeclareMathSymbol{\varspade}{\mathalpha}{extraup}{85}
\usepackage{graphicx}
\usepackage{adjustbox}
\usepackage{float}
\usepackage{xfrac}
\usepackage{rotating}
\usepackage{amsmath}
\usepackage{array}
\usepackage{calc}
\usepackage{caption}
\usepackage{multido,mathpazo}
\usepackage[usenames,dvipsnames]{pstricks}
\usepackage{epsfig}
\usepackage{pst-grad}
\usepackage{pst-plot} 
\usepackage{enumerate}
\usepackage{xcolor}
\usepackage{tikz}
\usepackage{mathpazo}
\usepackage{multido,mathpazo}
\usepackage{multido,mathpazo}
\usepackage[ruled]{algorithm2e}
\definecolor{light-gray}{gray}{0.9}

\newtheorem{conjecture}{Conjecture}
\newtheorem{lemma}{Lemma}
\newtheorem{proposition}{Proposition}

\newtheorem{remark}{Remark}

\newtheorem{definitiona}{Definition}

\newenvironment{proof}{{\bf Proof.}}{\hspace*{1mm}\hfill\rule{2mm}{2mm}}

\newtheorem{pretheorema}{{\bf Theorem}}

\newenvironment{theorema}[1]{\begin{pretheorema}
		{\hspace{-0.2em}{\rm #1}{\bf.}}}{\end{pretheorema}}
\newtheorem{prelemmab}{{\bf Lemma}}

\newtheorem{prepropositiona}{{\bf Proposition}}

\input amssym.tex
\def\n#1{\vbox to 3mm{\vspace{1mm}\vfill \hbox to 2.0mm{\hfill
			$#1$\hfill} \vfill }}
\def\m#1#2{\raise 0.2ex\hbox{
		${#1_{\bf \displaystyle #2}}$}}
\def\x#1{\raise 0.5ex\hbox{
		${#1}$}}

\title{\bf On decomposing complete tripartite graphs into 5-cycles}
\author{
	M. Abdolmaleki\thanks{Department of Mathematical Sciences, Sharif University of Technology,
		Tehran, I. R. Iran},
	S. Gh. Ilchi\thanks{Department of Computer Engineering,
		Sharif University of Technology,
		Tehran, I. R. Iran},
	E.S. Mahmoodian\footnotemark[1],
	and
	MA. Shabani\footnotemark[1]}
\date{}
\begin{document}
	\maketitle  
	
	\begin{abstract}
		The problem of finding necessary and sufficient conditions to decompose a complete tripartite graph $K_{r,s,t}$ into 5-cycles was first considered by E.S. Mahmoodian and Maryam Mirzakhani (1995). They stated some necessary conditions and conjectured that those conditions are also sufficient. Since then, many cases of the problem have been solved by various authors; however, the case when the partite sets $r\leq s\leq t$ have odd and distinct sizes remained open. We show the conjecture is true when $r$, $s$ and $t$ are all multiples of 5, $t+90 \leq \frac{4rs}{r+s}$, and $t \neq s+10$.
	\end{abstract}
	\maketitle
	\section{Introduction}
	An $m$-cycle decomposition of a graph $G$ is a partition of $E(G)$ into $m$-cycles. Various results have been obtained concerning $m$-cycle decompositions of complete graphs or complete multipartite graphs. The problem of decomposing complete $k$-partite graphs into $m$-cycles when $k=2$ is completely solved by Sotteau in \cite{MR609596}. Also in \cite{MR1010577}, it is shown that a complete $k$-partite graph with each part of size $m$ decomposes into $m$-cycles if and only if both $k$ and $m$ are odd. In case $k=3$, it is easy to solve the problem for $m=3$, Mahmoodian and Mirzakhani \cite{MR1366852}  introduced some necessary conditions for $m=5$ and showed that these conditions are sufficient for some cases.
	In the following theorem Mahmoodian and Mirzakhani have shown the necessary conditions for decomposing $K_{r,s,t}$ into 5-cycles. 
	{\begin{theorema}{\rm \cite{MR1366852}}
			\label{trmA}
			Suppose $r\leq s\leq t$ and $K_{r,s,t}$ can be decomposed into 5-cycles. Then the following conditions must hold:
			\begin{enumerate}
				\item $r$, $s$, and $t$ have the same parity,
				\item $5 | rs+rt+st$,
				\item $t\leq 4rs/(r+s)$.
			\end{enumerate}
		\end{theorema}
		\begin{conjecture}
			{\rm \cite{MR1366852}}.
			The necessary conditions in Theorem \ref{trmA} are also sufficient.	
		\end{conjecture}
		
		Mahmoodian and Mirzakhani \cite{MR1366852} also showed that if two parts of a complete tripartite graph have an equal number of vertices, say $K_{r,r,s}$, then the above conditions are also sufficient, except only for the case when $r$ is a multiple of $5$ but $s$ is not. Cavenagh and Billington \cite{MR1795321} extended this result to when two parts have equal size or $r$ and $s$ are both multiples of 10. They also showed in \cite{MR2855067} that when $r$, $s$ and $t$ are all odd and $r \leq s \leq t \leq \kappa r$, where $\kappa \approx 1.0806$, then the conjectured necessary conditions for decomposing are sufficient. In \cite{MR3013258} this result is further improved to the cases where $\kappa \approx \frac{5}{3}$. Cavenagh \cite{MR1927056} also showed that if each part of a graph has an even size, then the necesserly conditions of Theorem \ref{trmA} are sufficient. We show that the conjecture is true for $K_{r,s,t}$ when $r$, $s$, and $t$ are multiples of $5$, where $t+90 \leq \frac{4rs}{r+s}$ and $t \neq s+10$.
		\subsection{Latin representation and Trades}
		A method for decomposing a complete tripartite graph into 5-cycles was developed in \cite{MR1795321} which in fact extends the idea of the latin square. Latin representation of a complete tripartite graph $K_{r, s, t}$ with parts $R$, $S$, and $T$ (respectively), say $L$, is a latin rectangle $L'$ of order $s \times r$ based on $t$ elements together with a set of $t-r$ entries at the end of each row of $L'$, and a set of $t-s$ entries at the end of each column of $L'$.
		
		\begin{figure}[H]
			\psscalebox{0.4 0.4}
			{
				\begin{pspicture}(0,-6.2)(20.82,6.2)
				\definecolor{colour0}{rgb}{0.8,0.8,0.8}
				\pscustom[linecolor=black, linewidth=0.04]
				{
					\newpath
					\moveto(11.2,4.2)
				}
				\psframe[linecolor=black, linewidth=0.04, dimen=outer](7.6,4.6)(2.0,-4.2)
				\psframe[linecolor=black, linewidth=0.02, dimen=outer](4.4,4.6)(4.0,-4.2)
				\psframe[linecolor=black, linewidth=0.02, dimen=outer](7.6,1.0)(2.0,0.6)
				\rput[bl](4.0,5.0){\fontsize{16 pt}{3 pt} \selectfont $i$}
				\rput[bl](1.5,0.6){\fontsize{16 pt}{3 pt} \selectfont $j$}
				\psline[linecolor=black, linewidth=0.04](7.6,4.6)(12.0,4.6)(12.0,-4.2)(7.6,-4.2)
				\psline[linecolor=black, linewidth=0.04](2.0,-4.2)(2.0,-5.4)(7.6,-5.4)(7.6,-4.2)(7.6,-4.2)
				\psline[linecolor=black, linewidth=0.04](15.6,1.0)(15.6,-2.2)(15.6,-2.2)
				\psline[linecolor=black, linewidth=0.04](16.8,1.0)(16.8,-2.2)
				\psbezier[linecolor=black, linewidth=0.04](15.6,-2.2)(15.6,-2.7818182)(16.8,-2.7818182)(16.8,-2.2)
				\psbezier[linecolor=black, linewidth=0.04](16.8,1.0)(16.8,1.5818182)(15.6,1.5818182)(15.6,1.0)
				\psline[linecolor=black, linewidth=0.04](17.6,1.8)(17.6,-2.2)(17.6,-2.2)
				\psline[linecolor=black, linewidth=0.04](18.8,1.8)(18.8,-2.2)
				\psbezier[linecolor=black, linewidth=0.04](17.6,-2.2)(17.6,-2.9272728)(18.8,-2.9272728)(18.8,-2.2)
				\psbezier[linecolor=black, linewidth=0.04](18.8,1.8)(18.8,2.5272727)(17.6,2.5272727)(17.6,1.8)
				\psline[linecolor=black, linewidth=0.04](19.6,2.6)(19.6,-2.2)(19.6,-2.2)
				\psline[linecolor=black, linewidth=0.04](20.8,2.6)(20.8,-2.2)
				\psbezier[linecolor=black, linewidth=0.04](19.6,-2.2)(19.6,-3.0727272)(20.8,-3.0727272)(20.8,-2.2)
				\psbezier[linecolor=black, linewidth=0.04](20.8,2.6)(20.8,3.4727273)(19.6,3.4727273)(19.6,2.6)
				\pscircle[linecolor=black, linewidth=0.04, fillstyle=gradient, gradlines=2000, gradbegin=black, gradend=black, dimen=outer](20.2,2.2){0.2}
				\pscircle[linecolor=black, linewidth=0.04, fillstyle=gradient, gradlines=2000, gradbegin=black, gradend=black, dimen=outer](18.2,-2.2){0.2}
				\pscircle[linecolor=black, linewidth=0.04, fillstyle=gradient, gradlines=2000, gradbegin=black, gradend=black, dimen=outer](16.2,0.2){0.2}
				\psline[linecolor=black, linewidth=0.04](16.0,0.2)(18.0,-2.2)(18.0,-2.2)
				\rput[bl](16.0,1.8){\fontsize{15 pt}{3 pt} \selectfont $R$}
				\rput[bl](18.0,2.6){\fontsize{15 pt}{3 pt} \selectfont $S$}
				\rput[bl](20.0,3.4){\fontsize{15 pt}{3 pt} \selectfont $T$}
				\rput[bl](15.9,0.6){\fontsize{16 pt}{3 pt} \selectfont $r_i$}
				\rput[bl](17.9,-1.9){\fontsize{16 pt}{3 pt} \selectfont $s_j$}
				\rput[bl](19.9,2.5){\fontsize{16 pt}{3 pt} \selectfont $t_k$}
				\psline[linecolor=black, linewidth=0.04](18.4,-2.2)(20.4,2.2)
				\psline[linecolor=black, linewidth=0.04](16.0,0.2)(20.0,2.2)
				\psline[linecolor=black, linewidth=0.02, tbarsize=0.07055555cm 5.0]{|-|}(7.6,5.0)(12.0,5.0)(12.0,5.0)
				\rput[bl](9.5,5.2){\fontsize{22 pt}{3 pt} \selectfont t-r}
				\rput[bl](4.5,-6.3){\fontsize{22 pt}{3 pt} \selectfont r}
				\psline[linecolor=black, linewidth=0.02, tbarsize=0.07055555cm 5.0]{|-|}(2.0,-5.8)(7.6,-5.8)
				\psline[linecolor=black, linewidth=0.02, tbarsize=0.07055555cm 5.0]{|-|}(1.2,4.6)(1.2,-4.2)
				\rput[bl](0.2,0.6){\fontsize{22 pt}{3 pt} \selectfont s}
				\psline[linecolor=black, linewidth=0.02, tbarsize=0.07055555cm 5.0]{|-|}(1.2,-4.2)(1.2,-5.4)
				\rput[bl](0.0,-5.0){\fontsize{22 pt}{3 pt} \selectfont t-s}
				\psline[linecolor=black, linewidth=0.02](15.2,-0.4)(15.2,-0.5)(12.96,-0.5)(12.96,-0.6)(12.4,-0.4)(12.96,-0.2)(12.96,-0.3)(15.2,-0.3)(15.2,-0.4)
				\psframe[linecolor=black, linewidth=0.02, fillstyle=gradient, gradlines=2000, gradbegin=black, gradend=black, dimen=outer](4.02,1.0)(4.0,0.98)
				\psframe[linecolor=black, linewidth=0.02, fillstyle=solid,fillcolor=colour0, opacity=0.5, dimen=outer](4.4,1.0)(4.0,0.6)
				\rput[bl](3.98,0.6){\fontsize{15 pt}{3 pt} \selectfont $k$}
				\end{pspicture}}
			\centering
			\caption{Each entry of the latin rectangle $L'$ represents a triangle}
			\label{latin_model}
		\end{figure}
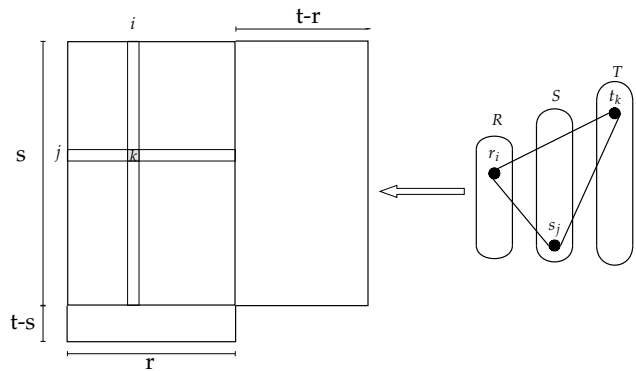
		Each entry of the $s \times r$ latin rectangle $L'$ represents a triangle in $K_{r,s,t}$, see Figure \ref{latin_model}. Each entry such as $k$ in the $i$-th column and $j$-th row of the $(t-s) \times r$ rectangle at the end of first $r$ columns represents the edge between parts $R$ and $T$, and each entry such as $k$ in the $i$-th column and $j$-th row of the $s \times (t-r)$ entry at the end of the $s$ first rows represents the edge between parts $S$ and $T$.\\
		In this paper, the cells of a latin representation are filled as a back circulant latin square of size $s$, with entries from 1 to $s$; and with entries $s + 1$ to $t$ in increasing order at the end of all rows and first $r$ columns. In the figure the notation of some entries have indices which will be explained later.
		
		\begin{figure}[H]
			\centering
			\begin{tikzpicture}[scale=.8]
			\begin{scope}
			\tiny
			\draw[step=1cm] (-8,-6) grid (8,6);
			\draw[step=1cm] (-8, -10) grid (-2, -6);
			\draw[very thick] (-8, -6) rectangle (-2, 6);
			\draw[very thick] (-2, -6) rectangle (4, 6);
			\draw[very thick] (4, -6) rectangle (8, 6);
			\draw[very thick] (-8, -10) rectangle (-2, -6);
			
			\node at (-7.50, 5.50) {1}; \node at (-6.50, 5.50) {$2_{1A}$}; \node at (-5.50, 5.50) {3}; \node at (-4.50, 5.50) {4}; \node at (-3.50, 5.50) {5}; \node at (-2.50, 5.50) {6}; \node at (-1.50, 5.50) {7}; \node at (-0.50, 5.50) {8}; \node at (0.50, 5.50) {9}; \node at (1.50, 5.50) {10}; \node at (2.50, 5.50) {11}; \node at (3.50, 5.50) {12}; \node at (4.50, 5.50) {$13_{1A}$}; \node at (5.50, 5.50) {$14_{1A}$}; \node at (6.50, 5.50) {15}; \node at (7.50, 5.50) {16}; 
			\node at (-7.50, 4.50) {$2_{1A}$}; \node at (-6.50, 4.50) {3}; \node at (-5.50, 4.50) {4}; \node at (-4.50, 4.50) {5}; \node at (-3.50, 4.50) {6}; \node at (-2.50, 4.50) {7}; \node at (-1.50, 4.50) {8}; \node at (-0.50, 4.50) {9}; \node at (0.50, 4.50) {10}; \node at (1.50, 4.50) {11}; \node at (2.50, 4.50) {12}; \node at (3.50, 4.50) {1}; \node at (4.50, 4.50) {$13_{1A}$}; \node at (5.50, 4.50) {$14_{1A}$}; \node at (6.50, 4.50) {15}; \node at (7.50, 4.50) {16}; 
			\node at (-7.50, 3.50) {3}; \node at (-6.50, 3.50) {$4_{1B}$}; \node at (-5.50, 3.50) {5}; \node at (-4.50, 3.50) {6}; \node at (-3.50, 3.50) {7}; \node at (-2.50, 3.50) {8}; \node at (-1.50, 3.50) {9}; \node at (-0.50, 3.50) {10}; \node at (0.50, 3.50) {11}; \node at (1.50, 3.50) {12}; \node at (2.50, 3.50) {1}; \node at (3.50, 3.50) {2}; \node at (4.50, 3.50) {$13_{1B}$}; \node at (5.50, 3.50) {$14_{1B}$}; \node at (6.50, 3.50) {15}; \node at (7.50, 3.50) {16}; 
			\node at (-7.50, 2.50) {$4_{1B}$}; \node at (-6.50, 2.50) {5}; \node at (-5.50, 2.50) {6}; \node at (-4.50, 2.50) {7}; \node at (-3.50, 2.50) {8}; \node at (-2.50, 2.50) {9}; \node at (-1.50, 2.50) {10}; \node at (-0.50, 2.50) {11}; \node at (0.50, 2.50) {12}; \node at (1.50, 2.50) {1}; \node at (2.50, 2.50) {2}; \node at (3.50, 2.50) {3}; \node at (4.50, 2.50) {13}; \node at (5.50, 2.50) {14}; \node at (6.50, 2.50) {15}; \node at (7.50, 2.50) {16}; 
			\node at (-7.50, 1.50) {$5_{1C}$}; \node at (-6.50, 1.50) {6}; \node at (-5.50, 1.50) {7}; \node at (-4.50, 1.50) {8}; \node at (-3.50, 1.50) {9}; \node at (-2.50, 1.50) {10}; \node at (-1.50, 1.50) {11}; \node at (-0.50, 1.50) {12}; \node at (0.50, 1.50) {1}; \node at (1.50, 1.50) {2}; \node at (2.50, 1.50) {3}; \node at (3.50, 1.50) {4}; \node at (4.50, 1.50) {$13_{1C}$}; \node at (5.50, 1.50) {$14_{1C}$}; \node at (6.50, 1.50) {15}; \node at (7.50, 1.50) {16}; 
			\node at (-7.50, 0.50) {$6_{1C}$}; \node at (-6.50, 0.50) {7}; \node at (-5.50, 0.50) {8}; \node at (-4.50, 0.50) {9}; \node at (-3.50, 0.50) {10}; \node at (-2.50, 0.50) {11}; \node at (-1.50, 0.50) {12}; \node at (-0.50, 0.50) {1}; \node at (0.50, 0.50) {2}; \node at (1.50, 0.50) {3}; \node at (2.50, 0.50) {4}; \node at (3.50, 0.50) {5}; \node at (4.50, 0.50) {$13_{1C}$}; \node at (5.50, 0.50) {$14_{1C}$}; \node at (6.50, 0.50) {15}; \node at (7.50, 0.50) {16}; 
			\node at (-7.50, -0.50) {$7$}; \node at (-6.50, -0.50) {$8_{1E}$}; \node at (-5.50, -0.50) {$9_{1E}$}; \node at (-4.50, -0.50) {$10_{1E}$}; \node at (-3.50, -0.50) {11}; \node at (-2.50, -0.50) {12}; \node at (-1.50, -0.50) {1}; \node at (-0.50, -0.50) {2}; \node at (0.50, -0.50) {3}; \node at (1.50, -0.50) {4}; \node at (2.50, -0.50) {5}; \node at (3.50, -0.50) {6}; \node at (4.50, -0.50) {$13$}; \node at (5.50, -0.50) {$14$}; \node at (6.50, -0.50) {15}; \node at (7.50, -0.50) {16}; 
			\node at (-7.50, -1.50) {$8$}; \node at (-6.50, -1.50) {9}; \node at (-5.50, -1.50) {10}; \node at (-4.50, -1.50) {11}; \node at (-3.50, -1.50) {12}; \node at (-2.50, -1.50) {1}; \node at (-1.50, -1.50) {2}; \node at (-0.50, -1.50) {3}; \node at (0.50, -1.50) {4}; \node at (1.50, -1.50) {5}; \node at (2.50, -1.50) {6}; \node at (3.50, -1.50) {7}; \node at (4.50, -1.50) {13}; \node at (5.50, -1.50) {$14$}; \node at (6.50, -1.50) {$15$}; \node at (7.50, -1.50) {16}; 
			\node at (-7.50, -2.50) {$9$}; \node at (-6.50, -2.50) {10}; \node at (-5.50, -2.50) {11}; \node at (-4.50, -2.50) {12}; \node at (-3.50, -2.50) {1}; \node at (-2.50, -2.50) {2}; \node at (-1.50, -2.50) {3}; \node at (-0.50, -2.50) {4}; \node at (0.50, -2.50) {5}; \node at (1.50, -2.50) {6}; \node at (2.50, -2.50) {7}; \node at (3.50, -2.50) {8}; \node at (4.50, -2.50) {$13$}; \node at (5.50, -2.50) {14}; \node at (6.50, -2.50) {$15$}; \node at (7.50, -2.50) {16}; 
			\node at (-7.50, -3.50) {10}; \node at (-6.50, -3.50) {11}; \node at (-5.50, -3.50) {12}; \node at (-4.50, -3.50) {1}; \node at (-3.50, -3.50) {2}; \node at (-2.50, -3.50) {3}; \node at (-1.50, -3.50) {4}; \node at (-0.50, -3.50) {5}; \node at (0.50, -3.50) {6}; \node at (1.50, -3.50) {7}; \node at (2.50, -3.50) {8}; \node at (3.50, -3.50) {9}; \node at (4.50, -3.50) {13}; \node at (5.50, -3.50) {14}; \node at (6.50, -3.50) {15}; \node at (7.50, -3.50) {16}; 
			\node at (-7.50, -4.50) {11}; \node at (-6.50, -4.50) {12}; \node at (-5.50, -4.50) {1}; \node at (-4.50, -4.50) {2}; \node at (-3.50, -4.50) {3}; \node at (-2.50, -4.50) {4}; \node at (-1.50, -4.50) {5}; \node at (-0.50, -4.50) {6}; \node at (0.50, -4.50) {7}; \node at (1.50, -4.50) {8}; \node at (2.50, -4.50) {9}; \node at (3.50, -4.50) {10}; \node at (4.50, -4.50) {13}; \node at (5.50, -4.50) {14}; \node at (6.50, -4.50) {15}; \node at (7.50, -4.50) {16}; 
			\node at (-7.50, -5.50) {12}; \node at (-6.50, -5.50) {1}; \node at (-5.50, -5.50) {2}; \node at (-4.50, -5.50) {3}; \node at (-3.50, -5.50) {4}; \node at (-2.50, -5.50) {5}; \node at (-1.50, -5.50) {6}; \node at (-0.50, -5.50) {7}; \node at (0.50, -5.50) {8}; \node at (1.50, -5.50) {9}; \node at (2.50, -5.50) {10}; \node at (3.50, -5.50) {11}; \node at (4.50, -5.50) {13}; \node at (5.50, -5.50) {14}; \node at (6.50, -5.50) {15}; \node at (7.50, -5.50) {16};   
			
			\node at (-7.50, -6.50) {$13_{1B}$}; \node at (-6.50, -6.50) {$13_{1E}$}; \node at (-5.50, -6.50) {13}; \node at (-4.50, -6.50) {$13_{1E}$}; \node at (-3.50, -6.50) {13}; \node at (-2.50, -6.50) {13}; 
			\node at (-7.50, -7.50) {$14_{1B}$}; \node at (-6.50, -7.50) {$14_{1E}$}; \node at (-5.50, -7.50) {$14_{1E}$}; \node at (-4.50, -7.50) {14}; \node at (-3.50, -7.50) {14}; \node at (-2.50, -7.50) {14}; 
			\node at (-7.50, -8.50) {15}; \node at (-6.50, -8.50) {15}; \node at (-5.50, -8.50) {$15_{1E}$}; \node at (-4.50, -8.50) {$15_{1E}$}; \node at (-3.50, -8.50) {15}; \node at (-2.50, -8.50) {15}; 
			\node at (-7.50, -9.50) {16}; \node at (-6.50, -9.50) {16}; \node at (-5.50, -9.50) {16}; \node at (-4.50, -9.50) {16}; \node at (-3.50, -9.50) {16}; \node at (-2.50, -9.50) {16}; 
			
			\end{scope}  
			
			\end{tikzpicture}
			\caption{A latin representation of $K_{6,12,16}$ and four trades in it}
			\label{fig:exp_trade}
		\end{figure}
		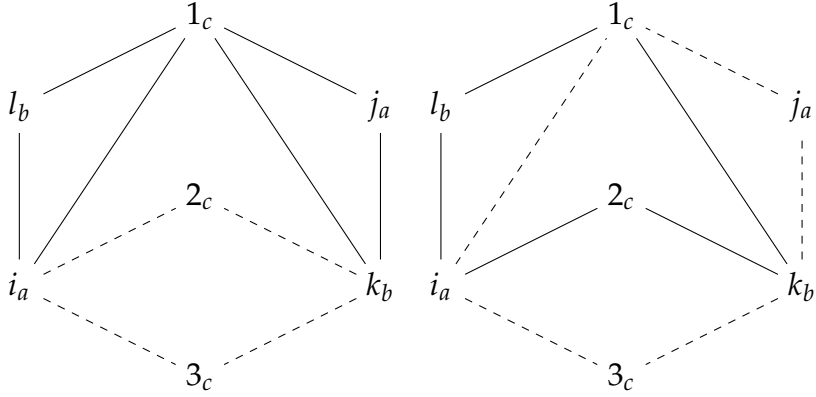
\begin{figure}[H]
			\centering
			\tikzstyle{vertex}=[circle,minimum size=20pt,inner sep=0pt]

			\begin{tikzpicture}[scale=1.2]
			\node[vertex] (ia) at (0, -1) {$i_a$};
			\node[vertex] (lb) at (0, 1) {$l_b$};
			\node[vertex] (1c) at (2, 2) {$1_c$};
			\node[vertex] (2c) at (2, 0) {$2_c$};
			\node[vertex] (3c) at (2, -2) {$3_c$};
			\node[vertex] (kb) at (4, -1) {$k_b$};
			\node[vertex] (ja) at (4, 1) {$j_a$};
			\draw (ia) -- (lb);
			\draw(ia) -- (1c);
			\draw(lb) -- (1c);
			\draw(kb) -- (ja);
			\draw(kb) -- (1c);
			\draw(ja) -- (1c);
			\draw[dashed] (ia) -- (2c);
			\draw[dashed] (ia) -- (3c);
			\draw[dashed] (kb) -- (2c);
			\draw[dashed] (kb) -- (3c);
			\end{tikzpicture}
			\begin{tikzpicture}[scale=1.2]
			\node[vertex] (ia) at (0, -1) {$i_a$};
			\node[vertex] (lb) at (0, 1) {$l_b$};
			\node[vertex] (1c) at (2, 2) {$1_c$};
			\node[vertex] (2c) at (2, 0) {$2_c$};
			\node[vertex] (3c) at (2, -2) {$3_c$};
			\node[vertex] (kb) at (4, -1) {$k_b$};
			\node[vertex] (ja) at (4, 1) {$j_a$};
			\draw (ia) -- (lb);
			\draw[dashed] (ia) -- (1c);
			\draw(kb) -- (1c);
			\draw[dashed] (kb) -- (ja);
			\draw(lb) -- (1c);
			\draw[dashed] (ja) -- (1c);
			\draw(ia) -- (2c);
			\draw[dashed] (ia) -- (3c);
			\draw(kb) -- (2c);
			\draw[dashed] (kb) -- (3c);
			\end{tikzpicture}
			\caption{An illustration of a trade $1B$}
			\label{trade1b}
		\end{figure}
		
		\begin{definitiona}
			Let $M$ be a latin representation of the complete tripartite graph $K_{r, s, t}$. A \begin{sf}trade\end{sf} is a set of entries in $M$, corresponding to a set of triangles and edges in $K_{r, s, t}$ which can be decomposed into $5$-cycles, see Table \ref{trd_table}. In the figure indices $a$, $b$, and $c$ represent the parts $R$, $S$, and $T$ respectively, which the vertex belongs to.
		\end{definitiona}
		
		Two groups of trades were introduced in \cite{MR1795321}. The first group made by some triangles and single edges but the second group made only by triangles. Some of the trades that we need are shown in Table \ref{trd_table}.
		
		\begin{table}
			\centering
			\begin{tabular}{ | c | c | c | }\hline
				Trade & Entries used in latin & Exchange of edges in complete \\ type & representation & tripartite graphs with 5-cycles \\ \hline
				\textbf{1A} & 
				\noindent\parbox[c]{3cm}{
					\begin{pspicture}(0,0)(3,2)
					\definecolor{colour0}{rgb}{0.0,0.0,0.2}
					\psframe[linecolor=black, linewidth=0.04, dimen=outer](0.5,0.5)(2.5,1.5)
					\psline[linecolor=black, linewidth=0.04, doubleline=true, doublesep=0.02](1.5,0.5)(1.5,1.5)
					\rput[bl](0.2,0.7){j}
					\rput[bl](0.2,1.2){i}
					\rput[bl](0.65,1.6){k}
					\rput[bl](1.25,1.6){l}
					\rput[bl](1.15,1.13){1}
					\rput[bl](0.65,0.63){1}
					\rput[bl](1.65,0.63){2}
					\rput[bl](2.15,0.63){3}
					\rput[bl](1.65,1.13){2}
					\rput[bl](2.15,1.13){3}
					\end{pspicture}
				}
				& \begin{tabular}[x]{@{}c@{}}$(i_a,l_b,1_c),(j_a,k_b,1_c),(i_a,2_c,j_a,3_c)$\\$ \iff (i_a,l_b,1_c,j_a,2_c),(i_a,1_c,k_b,j_a,3_c)$\end{tabular}
				\\ \hline
				\textbf{1B} & 
				\noindent\parbox[c]{3cm}{
					\begin{pspicture}(0,0)(3,3)
					\definecolor{colour0}{rgb}{0.0,0.0,0.2}
					\psframe[linecolor=black, linewidth=0.04, dimen=outer](0.5,1.5)(2.5,2.5)
					\psframe[linecolor=black, linewidth=0.04, dimen=outer](0.5,0.5)(1.5,1.5)
					\psline[linecolor=black, linewidth=0.04, doubleline=true, doublesep=0.02](1.5,1.5)(1.5,2.5)
					\psline[linecolor=black, linewidth=0.04, doubleline=true, doublesep=0.02](0.5,1.5)(1.5,1.5)
					\rput[bl](0.2,1.7){j}
					\rput[bl](0.2,2.2){i}
					\rput[bl](0.65,2.6){k}
					\rput[bl](1.15,2.6){l}
					\rput[bl](1.15,2.13){1}
					\rput[bl](0.65,1.63){1}
					\rput[bl](1.65,2.13){2}
					\rput[bl](2.15,2.13){3}
					\rput[bl](0.65,1.11){2}
					\rput[bl](0.65,0.61){3}
					\end{pspicture}
				}
				& \begin{tabular}[x]{@{}c@{}}$(i_a,l_b,1_c),(j_a,k_b,1_c),(i_a,2_c,k_b,3_c)$\\$ \iff (i_a,l_b,1_c,k_b,2_c),(i_a,1_c,j_a,k_b,3_c)$\end{tabular}
				\\ \hline
				\textbf{1C} & 
				\noindent\parbox[c]{3cm}{
					\begin{pspicture}(0,0)(2.5,2)
					\definecolor{colour0}{rgb}{0.0,0.0,0.2}
					\psframe[linecolor=black, linewidth=0.04, dimen=outer](0.5,0.5)(2.0,1.5)
					\psline[linecolor=black, linewidth=0.04, doubleline=true, doublesep=0.02](1.0,0.5)(1.0,1.5)
					\rput[bl](0.2,0.7){j}
					\rput[bl](0.2,1.2){i}
					\rput[bl](0.65,1.6){k}
					\rput[bl](0.65,1.13){1}
					\rput[bl](0.65,0.63){2}
					\rput[bl](1.15,0.63){3}
					\rput[bl](1.65,0.63){4}
					\rput[bl](1.15,1.13){3}
					\rput[bl](1.65,1.13){4}
					\end{pspicture}
				}
				& \begin{tabular}[x]{@{}c@{}}$(i_a,k_b,1_c),(j_a,k_b,2_c),(i_a,3_c,j_a,4_c)$\\$ \iff (i_a,k_b,2_c,j_a,3_c),(i_a,1_c,k_b,j_a,4_c)$\end{tabular}
				\\ \hline
				\textbf{1E} & 
				\noindent\parbox[c]{3cm}{
					\begin{pspicture}(0,0)(2.5,2.5)
					\definecolor{colour0}{rgb}{0.0,0.0,0.2}
					\psframe[linecolor=black, linewidth=0.04, dimen=outer](0.5,0.5)(2.0,2.0)
					\psline[linecolor=black, linewidth=0.04, doubleline=true, doublesep=0.02](0.5,1.5)(2.0,1.5)
					\rput[bl](0.65,2.07){j}
					\rput[bl](1.15,2.12){k}
					\rput[bl](1.65,2.12){l}
					\rput[bl](0.23,1.62){i}
					\rput[bl](0.65,1.63){1}
					\rput[bl](1.15,1.63){2}
					\rput[bl](1.65,1.63){3}
					\rput[bl](0.65,1.1){4}
					\rput[bl](1.15,1.1){5}
					\rput[bl](1.65,1.1){6}
					\rput[bl](0.65,0.63){5}
					\rput[bl](1.15,0.63){6}
					\rput[bl](1.65,0.63){4}
					\end{pspicture}
				}
				& \begin{tabular}[x]{@{}c@{}}$(i_a,j_b,1_c),(i_a,k_b,2_c),(i_a,l_b,3_c),(j_b,4_c,l_b,6_c,k_b,5_c)$\\$ \iff (i_a,1_c,j_b,4_c,l_b),(i_a,2_c,k_b,5_c,j_b),(i_a,3_c,l_b,6_c,k_b)$\end{tabular}
				\\ \hline
			\end{tabular}
			\caption{Some trades of the first type}
			\label{trd_table}
		\end{table}
		Figure \ref{fig:exp_trade} and \ref{trade1b} show respectively the latin representation of $K_{6,12,16}$, and how a trade of type $1B$ can be decomposed into 5-cycles.
		In this paper, we use these trades for decomposing complete tripartite graphs into 5-cycles.
		\section{Transforming the problem}
		In the next section, we prove that existence of a decomposition of $K_{r,s,t}$ into triangles and 5-cycles implies that $K_{5r,5s,5t}$ can be decomposed into 5-cycles. For proving this statement, we state some related definitions and lemmas.
		
		\begin{definitiona}
			Assume $G(V,E)$ is a simple graph. We define a \begin{sf}k-blowup\end{sf} operation on $G$ as replacing each vertex of $G$ with $k$ vertices and drawing an edge between two vertices in the new graph if and only if their corresponding vertices in $G$ are adjacent. We denote the resulting graph as
			\begin{sf}
				$B_{k}(G)$
			\end{sf}.
		\end{definitiona}
		\begin{lemma}
			\label{DecompositionOfCycles}
			$B_{5}(C_{3})$ and $B_{5}(C_{5})$ can be decomposed into 5-cycles. 
		\end{lemma}
		\begin{proof}
			$B_{5}(C_{3})$ is equivalent to $K_{5,5,5}$. By \cite{MR1366852}, we know that $K_{5,5,5}$ can be decomposed into 5-cycles. 
			In order to decompose $B_{5}(C_{5})$, assume $A_{i} = \{a_{i,1},a_{i,2},...,a_{i,5}\}$ be a set of vertices corresponding to vertex $i$ of $C_{5}$ after {\em 5-blowup}
			operation. With the following base cycles, we can decompose $B_{5}(C_{5})$:\\
			\begin{equation*}
				(a_{1,1+i}, a_{2,1+i}, a_{3,1+i}, a_{4,1+i}, a_{5,1+i})
			\end{equation*}
			\begin{equation*}
				(a_{1,1+i}, a_{2,2+i}, a_{3,3+i}, a_{4,4+i}, a_{5,5+i})
			\end{equation*}
			\begin{equation*}
				(a_{1,1+i}, a_{2,3+i}, a_{3,5+i}, a_{4,2+i}, a_{5,4+i})
			\end{equation*}
			\begin{equation*}
				(a_{1,1+i}, a_{2,4+i}, a_{3,2+i}, a_{4,5+i}, a_{5,3+i})
			\end{equation*}
			\begin{equation*}
				(a_{1,1+i}, a_{2,5+i}, a_{3,4+i}, a_{4,3+i}, a_{5,2+i})
			\end{equation*}
			by considering $i = 0,...,4$ and assuming the calculations in mod 5, a decomposition of $B_{5}(C_{5})$ will be obtained. 
		\end{proof}
		\begin{proposition}
			If all three necessary conditions suggested in Theorem ~\ref{trmA} hold for $K_{5r,5s,5t}$, then the first and the third conditions also hold for $K_{r,s,t}$.
			\label{pp:conditions}
		\end{proposition}
		The proof of the following lemma can be obtained by some simple calculations.
		\begin{lemma}
			The existence of decomposition of $K_{r,s,t}$ into triangles and 5-cycles, implies that $K_{5r,5s,5t}$ can be decomposed into 5-cycles.
		\end{lemma}
		\begin{proof}
			We know that $K_{5r,5s,5t}$ is equivalent to $B_{5}(K_{r,s,t})$. The existence of decomposition of $K_{5r,5s,5t}$ into 5-cycles is deduced from Lemma ~\ref{DecompositionOfCycles} and a decomposition of $K_{r,s,t}$ into triangles and 5-cycles.
		\end{proof}
		\section{Preliminaries}
		As mentioned in the previous section, in order to decompose $K_{5r,5s,5t}$ into 5-cycles, it is enough to show that if three mentioned conditions are satisfied for a complete tripartite graph $K_{5r,5s,5t}$, then we can decompose $K_{r,s,t}$ into triangles and 5-cycles. \\
		In \cite{MR1927056} Cavenagh proved the main conjecture where the parameters $r$, $s$, and $t$ are all even. Therefore, in this section, we just consider the case where all parts of the graph have an odd size.
		Additionally, based on Proposition \ref{pp:conditions}, we know that if all the three necessary conditions were held for $K_{5r,5s,5t}$, then the first and the third conditions hold for $K_{r,s,t}$. 
		For any two integers $x$ and $y$, with a slight abuse of notation, we use the expression $x \pmod y $ to refer to the integer $a$ that satisfies the congruence relation $x \equiv a (\bmod~\: y)$ as well as an implicit relation $0 \leq a \leq y-1$.
		
		\begin{definitiona}
			In a latin representation of $G$ An \begin{sf}even cell\end{sf} is a cell with an even entry, and an \begin{sf}odd cell\end{sf} is a cell with an odd entry.
		\end{definitiona}
		We will show that a graph $K_{r,s,t}$ satisfying the given conditions, can be decomposed into 3 and 5 cycles.
		To do so, we partition the latin representation of $K_{r,s,t}$ into four parts $A$, $B$, $C$, and $D$ as follows:
		\begin{enumerate}
			\item Part $A$: All cells $(i,j)$, where $1\le i\le s$ and $1\le j\le r$.
			\item Part $B$: All cells $(i,j)$, where $1\le i\le s$ and $r+1\le j\le s$.
			\begin{itemize}
				\item Part $B_1$: Union of the even cells in ($r+1$)th column, the odd cells in ($s-1$)th column, cells in ($s$)th column and all cells with entry $s$ in part $B$.
				\item Part $B_2$: $B \setminus B_1$
			\end{itemize}
			\item Part $C$: All cells $(i,j)$, where $1\le i\le s$ and $s+1\le j\le t$.
			\begin{itemize}
				\item Part $C_1$: All cells in columns $(s+1)$ through $(t-4)$.
				\item Part $C_2$: $C \setminus C_1$
			\end{itemize}
			\item Part $D$: All cells $(i,j)$, where $s+1\le i\le t$ and $1\le j\le r$.
			\begin{itemize}
				\item Part $D_1$: All cells in rows $(s+1)$ through $(t-4)$.
				\item Part $D_2$: $D \setminus D_1$
			\end{itemize}
		\end{enumerate}
		
		\begin{figure}[H]
			\psscalebox{1.3 1.4}
			{
				\begin{pspicture}(0,0)(8,8)
				\psframe[linecolor=black, linewidth=0.04, dimen=outer](0.5,2.5)(7.5,7.1)
				\psframe[linecolor=black, linewidth=0.04, dimen=outer](0.5,0.5)(3.3,2.5)
				\psline[linecolor=black, linewidth=0.04](5.4,2.5)(5.4,7.1)
				
				\psline[linecolor=black, linewidth=0.02](0.85,0.5)(0.85,7.1)
				\psline[linecolor=black, linewidth=0.02](1.2,0.5)(1.2,7.1)
				
				\psline[linecolor=black, linewidth=0.02](2.6,0.5)(2.6,7.1)
				\psline[linecolor=black, linewidth=0.02](2.95,0.5)(2.95,7.1)
				
				\psline[linecolor=black, linewidth=0.02](3.65,2.5)(3.65,7.1)
				\psline[linecolor=black, linewidth=0.02](4,2.5)(4,7.1)
				
				\psline[linecolor=black, linewidth=0.02](4.7,2.5)(4.7,7.1)
				\psline[linecolor=black, linewidth=0.02](5.05,2.5)(5.05,7.1)
				
				\psline[linecolor=black, linewidth=0.04](5.4,2.5)(5.4,7.1)
				\psline[linecolor=black, linewidth=0.02](5.75,2.5)(5.75,7.1)
				\psline[linecolor=black, linewidth=0.02](6.1,2.5)(6.1,7.1)
				\psline[linecolor=black, linewidth=0.02](6.8,2.5)(6.8,7.1)
				\psline[linecolor=black, linewidth=0.02](7.15,2.5)(7.15,7.1)
				\psline[linecolor=black, linewidth=0.02](0.5,6.75)(7.5,6.75)
				\psline[linecolor=black, linewidth=0.02](0.5,6.4)(7.5,6.4)
				
				\psline[linecolor=black, linewidth=0.02](0.5,2.85)(7.5,2.85)
				\psline[linecolor=black, linewidth=0.02](0.5,3.2)(7.5,3.2)
				
				\psline[linecolor=black, linewidth=0.02](0.5,2.15)(3.3,2.15)
				\psline[linecolor=black, linewidth=0.02](0.5,1.8)(3.3,1.8)
				
				\psline[linecolor=black, linewidth=0.02](0.5,0.85)(3.3,0.85)
				\psline[linecolor=black, linewidth=0.02](0.5,1.2)(3.3,1.2)
				
				\psline[linecolor=black, linewidth=0.03, doubleline=true, doublesep=0.01](0.5,2.5)(3.3,2.5)
				\psline[linecolor=black, linewidth=0.03, doubleline=true, doublesep=0.01](3.3,2.53)(3.3,7.1)
				
				\rput[bl](0.64,6.85){\fontsize{5.5 pt}{3 pt}\selectfont 1}
				\rput[bl](0.99,6.85){\fontsize{5.5 pt}{3 pt}\selectfont 2}
				
				\rput[bl](0.64,6.5){\fontsize{5.5 pt}{10 pt}\selectfont 2}
				\rput[bl](0.99,6.5){\fontsize{5.5 pt}{10 pt}\selectfont 3}
				
				\rput[bl](2.65,6.85){\fontsize{5.5 pt}{3 pt}\selectfont r-1}
				\rput[bl](3.08,6.85){\fontsize{5.5 pt}{3 pt}\selectfont r}
				
				\rput[bl](2.7,6.5){\fontsize{5.5 pt}{3 pt}\selectfont r}
				\rput[bl](2.98,6.5){\fontsize{5.5 pt}{3 pt}\selectfont r+1}
				
				\rput[bl](3.35,6.85){\fontsize{5.5 pt}{3 pt}\selectfont r+1}
				\rput[bl](3.68,6.85){\fontsize{5.5 pt}{3 pt}\selectfont r+2}
				
				\rput[bl](3.35,6.5){\fontsize{5.5 pt}{3 pt}\selectfont r+2}
				\rput[bl](3.68,6.5){\fontsize{5.5 pt}{3 pt}\selectfont r+3}
				
				\rput[bl](4.75,6.85){\fontsize{5.5 pt}{3 pt}\selectfont s-1}
				\rput[bl](5.15,6.85){\fontsize{5.5 pt}{3 pt}\selectfont s}
				
				\rput[bl](4.85,6.5){\fontsize{5.5 pt}{3 pt}\selectfont s}
				\rput[bl](5.15,6.5){\fontsize{5.5 pt}{3 pt}\selectfont 1}
				
				\rput[bl](5.45,6.85){\fontsize{5.5 pt}{3 pt}\selectfont s+1}
				\rput[bl](5.78,6.85){\fontsize{5.5 pt}{3 pt}\selectfont s+2}
				
				\rput[bl](5.45,6.5){\fontsize{5.5 pt}{3 pt}\selectfont s+1}
				\rput[bl](5.78,6.5){\fontsize{5.5 pt}{3 pt}\selectfont s+2}
				
				\rput[bl](6.87,6.85){\fontsize{5.5 pt}{3 pt}\selectfont t-1}
				\rput[bl](7.3,6.85){\fontsize{5.5 pt}{3 pt}\selectfont t}
				
				\rput[bl](6.87,6.5){\fontsize{5.5 pt}{3 pt}\selectfont t-1}
				\rput[bl](7.3,6.5){\fontsize{5.5 pt}{3 pt}\selectfont t}

				\rput[bl](5.45,2.98){\fontsize{5.5 pt}{3 pt}\selectfont s+1}
				\rput[bl](5.78,2.98){\fontsize{5.5 pt}{3 pt}\selectfont s+2}
				
				\rput[bl](5.45,2.63){\fontsize{5.5 pt}{3 pt}\selectfont s+1}
				\rput[bl](5.78,2.63){\fontsize{5.5 pt}{3 pt}\selectfont s+2}
				
				\rput[bl](6.87,2.98){\fontsize{5.5 pt}{3 pt}\selectfont t-1}
				\rput[bl](7.3,2.98){\fontsize{5.5 pt}{3 pt}\selectfont t}
				
				\rput[bl](6.87,2.63){\fontsize{5.5 pt}{3 pt}\selectfont t-1}
				\rput[bl](7.3,2.63){\fontsize{5.5 pt}{3 pt}\selectfont t}

				\rput[bl](3.37,2.98){\fontsize{5.5 pt}{3 pt}\selectfont r-1}
				\rput[bl](3.75,2.98){\fontsize{5.5 pt}{3 pt}\selectfont r}
				
				\rput[bl](3.42,2.63){\fontsize{5.5 pt}{3 pt}\selectfont r}
				\rput[bl](3.68,2.63){\fontsize{5.5 pt}{3 pt}\selectfont r+1}
				
				\rput[bl](4.76,2.98){\fontsize{5.5 pt}{3 pt}\selectfont s-3}
				\rput[bl](5.12,2.98){\fontsize{5.5 pt}{3 pt}\selectfont s-2}
				
				\rput[bl](4.78,2.63){\fontsize{5.5 pt}{3 pt}\selectfont s-2}
				\rput[bl](5.12,2.63){\fontsize{5.5 pt}{3 pt}\selectfont s-1}

				\rput[bl](0.59,2.98){\fontsize{5.5 pt}{3 pt}\selectfont s-1}
				\rput[bl](0.99,2.98){\fontsize{5.5 pt}{3 pt}\selectfont s}
				
				\rput[bl](0.65,2.63){\fontsize{5.5 pt}{10 pt}\selectfont s}
				\rput[bl](0.99,2.63){\fontsize{5.5 pt}{10 pt}\selectfont 1}
				
				\rput[bl](2.66,2.98){\fontsize{5.5 pt}{3 pt}\selectfont r-3}
				\rput[bl](3.02,2.98){\fontsize{5.5 pt}{3 pt}\selectfont r-2}
				
				\rput[bl](2.68,2.63){\fontsize{5.5 pt}{3 pt}\selectfont r-2}
				\rput[bl](3,2.63){\fontsize{5.5 pt}{3 pt}\selectfont r-1}
				
				\rput[bl](0.56,2.28){\fontsize{5.5 pt}{3 pt}\selectfont s+1}
				\rput[bl](0.91,2.28){\fontsize{5.5 pt}{3 pt}\selectfont s+1}
				
				\rput[bl](0.54,1.93){\fontsize{5.5 pt}{10 pt}\selectfont s+2}
				\rput[bl](0.88,1.93){\fontsize{5.5 pt}{10 pt}\selectfont s+2}
				
				\rput[bl](2.64,2.28){\fontsize{5.5 pt}{3 pt}\selectfont s+1}
				\rput[bl](2.98,2.28){\fontsize{5.5 pt}{3 pt}\selectfont s+1}
				
				\rput[bl](2.64,1.93){\fontsize{5.5 pt}{3 pt}\selectfont s+2}
				\rput[bl](2.97,1.93){\fontsize{5.5 pt}{3 pt}\selectfont s+2}
				
				\rput[bl](0.56,0.97){\fontsize{5.5 pt}{3 pt}\selectfont t-1}
				\rput[bl](0.91,0.97){\fontsize{5.5 pt}{3 pt}\selectfont t-1}
				
				\rput[bl](0.64,0.63){\fontsize{5.5 pt}{10 pt}\selectfont t}
				\rput[bl](0.98,0.63){\fontsize{5.5 pt}{10 pt}\selectfont t}
				
				\rput[bl](2.64,0.97){\fontsize{5.5 pt}{3 pt}\selectfont t-1}
				\rput[bl](2.98,0.97){\fontsize{5.5 pt}{3 pt}\selectfont t-1}
				
				\rput[bl](2.74,0.63){\fontsize{5.5 pt}{3 pt}\selectfont t}
				\rput[bl](3.07,0.63){\fontsize{5.5 pt}{3 pt}\selectfont t}
				
				\rput[bl](1.7,5.2){\fontsize{13 pt}{1 pt}\color {gray} \selectfont A}
				\rput[bl](4.2,5.2){\fontsize{13 pt}{1 pt}\color {gray}\selectfont B}
				\rput[bl](6.3,5.2){\fontsize{13 pt}{1 pt}\color {gray}\selectfont C}
				\rput[bl](1.7,1.35){\fontsize{13 pt}{1 pt}\color {gray}\selectfont D}

				\rput[bl](0,1.4){\fontsize{12 pt}{3 pt}\selectfont t-s}
				\rput[bl](0.2,5.0){\fontsize{12 pt}{3 pt}\selectfont s}
				\rput[bl](1.85,7.2){\fontsize{12 pt}{3 pt}\selectfont r}
				\rput[bl](4.15,7.2){\fontsize{12 pt}{3 pt}\selectfont s-r}
				\rput[bl](6.25,7.2){\fontsize{12 pt}{3 pt}\selectfont t-s}
				
				\end{pspicture}
			}
			\centering
			\caption{Different parts and entries in latin Representation}
			\label{exp_table}
		\end{figure}
		
		Now, if we cover some cells of part $A$ and all cells in parts $B$, $C$, and $D$ by non-overlapped trades, then uncovered cells of part $A$ (if any) are triangles themselves, and we have a decomposition of $K_{r,s,t}$ into 5-cycles and triangles.
		\begin{definitiona}
			In a latin representation of $G$ a \begin{sf}$k$-cell\end{sf} is a cell with entry $k$.
		\end{definitiona}
		\begin{definitiona}\label{def:kdiag}
			The \begin{sf}$k-$diagonal \end{sf} is the set of all cells with entry $k$ in part $A$.
		\end{definitiona}
		
		As we will see in the next sections, we have to define an ordering on the $k-$diagonals that is not the same as the numbers in the Latin representation. To do so, we assign labels to $k-$diagonals by introducing the bijective function $L: \{1, \ldots, s\} \rightarrow \{1, \ldots, s\}$, where $L(i)$ is the label associated with the $i-$diagonal. Since $L$ is a bijective function, $L^{-1}(i)$ is unique, which represents the entry of each cell in the diagonal with label $i$. For the sake of simplicity, we define $L^{-1}$. $L$ can be uniquely determined by inverting $L^{-1}$. Note that all the operations are computed modulo $s$.
		
		\[ L^{-1}(i) = 
		\begin{cases} 
		r+2 & i = 1 \\
		L^{-1}(i-1) + r - 3  & L^{-1}(i-1) + r - 3  \notin \{ L^{-1}(1), \ldots, L^{-1}(i-1) \} \\
		L^{-1}(i-1) + r - 5  & otherwise
		\end{cases}
		\]
		
		\begin{lemma}
			All the diagonals will be labeled by running the above procedure.
		\end{lemma}
		\begin{proof}
			Let $d = gcd(r-3,s)$. Divide diagonals into $d$ groups such that diagonal $i$ is a member of group $i\pmod d$($1\leq i\leq s$); therefore, size of each group is equal to $\frac{s}{d}$. On the grounds that $gcd(\frac{r-3}{d},\frac{s}{d})=1$, when we visit the first diagonal in an arbitrary group, in the next $\frac{s}{d}-1$ steps we visit all the diagonals in the same group and then we go back to the first visited diagonal of the group. Hence, if we label one diagonal of a group, all members of that group will be labeled. Additionally, since $gcd(d,2)=1$, it can be proved similarly that all the diagonals will be labeled.
		\end{proof}
		
		\begin{lemma}
			\label{rslimit}
			For each triple $(r,s,t)$ where $r\leq s\leq t$ and $t+18 \leq \frac{4rs}{r+s}$, we have $18\leq r,s$.
		\end{lemma}
		\begin{proof}
			\begin{equation*}
				t+18 \leq \frac{4rs}{r+s} \leq r+s 
			\end{equation*}
			
			Where the last inequality is the AM-GM inequality for the numbers $r$ and $s$. 
			
			\begin{equation*}
				\Rightarrow 18\leq r,s
			\end{equation*}

		\end{proof}
		
		\begin{remark}
			In the next sections, all the calculations about rows' number are calculated in $\mathbb{N}_s$ except when we talk about part $D$.
		\end{remark}
		In the next sections, we use the above definitions and lemmas in order to refer to different parts of the latin representation and specific cells of it. 
		\section{Partial-trades}\label{sec::partial-trades}
		In this section, we cover parts $B$, $C$, and $D$ by partial-trades, and in the next section we will expand them to trades in the direction of proving the main theorem.
		\begin{definitiona}
			A \begin{sf}partial-trade\end{sf} is a group of cells in a latin representation that can be expanded to a trade together with some cells in part $A$.
		\end{definitiona}
		\subsection{Covering $B_1$ by partial-trades}
		In order to cover $B_1$ by partial-trades, we pair them up with some cells from part $C_2$ to form partial trades. For each odd number $i$ $(1\leq i\leq s-1)$, we take the two cells with entry $i$ in $B_1$ and the $(t-2)$-cells in their corresponding rows to form a partial-trade of type $1A$. All the odd cells of $B_1$ are in the different rows, so there is no intersection between these partial-trades. 
		
		On the contrary, it is possible that we have two cells with different even entry at the same row in part $B_1$. As such, we separate the even cells in part $B_1$ into two groups where in each part there is no two different cells even entry in the same row. In the first group, For each even number $i$ $(1\leq i\leq s)$ satisfying $1\le i\le r+1$ or $2r+4\le i\le s$, we consider two cells in $B_1$ with entry $i$ and the $t$-cells in their rows as a partial-trade. In the second group, for each even number satisfying $r+3\le i\le min\{ 2r+3,s\}$ we consider two cells in $B_1$ with entry $i$ and $(t-1)$-cells in corresponding rows as a partial-trade of type $1C$.
		
		Finally, in order to cover $s$-cells of $B_{1}$, for each $i$ $(1\leq i\leq \frac{s-r}{2})$, we create a partial-trade of type $1A$ by considering $s$-cells of $B_{1}$ in rows $2i-1$ and $2i$, and $(t-3)$-cells of column $t-3$ in rows $2i-1$ and $2i$. Since $s-r$ is even, all the $s$-cells of $B_{1}$ will be covered.
		
		As we have decomposed the part $B_1$ into smaller sets and then covered each of them with partial trades, we can conclude that there is no intersection between the mentioned partial trades in part $B_1$. Moreover, we have used cells from different columns of part $C_2$ to pair up with cells from different subsets of part $B_1$- i.e., $(t-2)$-cells for the set of odd sells in $B_1$, $t$-cells and $(t-1)$-cells for the two group of even cells in $B_1$, and $(t-3)$-cells the set of $s$-cells in $B_1$. As such, we can conclude that there is also no intersection between the mentioned partial trades in part $C_2$.

		\subsection{Covering $B_2$ by partial-trades}
		In order to cover $B_{2}$ by partial-trades, for each odd $i$ $(1\leq i\leq s-1)$, consider all cells with entry $i$ and $i+1$ in part $B_{2}$. These cells are in consecutive rows from $i+3$ to $(s-r)+i$. As such, for each $(2\leq j \leq \frac{s-r}{2})$ we create a partial-trade of type $1A$ with cells of rows $i+2j-1$ and $i+2j$.
		
		\begin{figure}[H]
			\psscalebox{0.7 0.7}
			{
				\begin{pspicture}(0,-4.275)(5.74,4.275)
				\definecolor{colour0}{rgb}{0.7,0.7,0.7}
				\definecolor{colour1}{rgb}{0.9,0.9,0.9}
				\psframe[linecolor=black, linewidth=0.04, fillstyle=solid,fillcolor=colour0, dimen=outer](5.2,2.925)(4.8,2.525)
				\psframe[linecolor=black, linewidth=0.04, fillstyle=solid,fillcolor=colour0, dimen=outer](5.2,2.525)(4.4,2.125)
				\psline[linecolor=black, linewidth=0.04](4.8,2.925)(4.8,2.125)(4.8,2.125)
				\psframe[linecolor=black, linewidth=0.04, fillstyle=solid,fillcolor=colour1, dimen=outer](4.8,2.125)(4.0,1.725)
				\psframe[linecolor=black, linewidth=0.04, fillstyle=solid,fillcolor=colour1, dimen=outer](4.4,1.725)(3.6,1.325)
				\psline[linecolor=black, linewidth=0.04](4.4,2.125)(4.4,1.725)(4.4,1.725)
				\psline[linecolor=black, linewidth=0.04](4.0,1.725)(4.0,1.325)
				\psframe[linecolor=black, linewidth=0.04, fillstyle=solid,fillcolor=colour0, dimen=outer](4.0,1.325)(3.2,0.925)
				\psframe[linecolor=black, linewidth=0.04, fillstyle=solid,fillcolor=colour0, dimen=outer](3.6,0.925)(2.8,0.525)
				\psline[linecolor=black, linewidth=0.04](3.6,0.925)(3.6,1.325)
				\psline[linecolor=black, linewidth=0.04](3.2,0.525)(3.2,0.925)
				\psframe[linecolor=black, linewidth=0.04, dimen=outer](5.2,3.725)(0.4,-4.275)
				\psframe[linecolor=black, linewidth=0.04, dimen=outer](1.2,-1.475)(0.4,-1.875)
				\psframe[linecolor=black, linewidth=0.04, dimen=outer](1.6,-1.075)(0.8,-1.475)
				\psline[linecolor=black, linewidth=0.04](1.2,-1.075)(1.2,-1.475)
				\psline[linecolor=black, linewidth=0.04](0.8,-1.475)(0.8,-1.875)
				\psframe[linecolor=black, linewidth=0.04, fillstyle=solid,fillcolor=colour1, dimen=outer](0.8,-1.875)(0.4,-2.275)
				\rput[bl](2.5,4.0){\fontsize{13 pt}{1 pt}\color {black} \selectfont s-r}
				\rput[bl](0.0,0.525){\fontsize{13 pt}{1 pt}\color {black} \selectfont s}
				\rput[bl](5.3,2.625){i+1}
				\rput[bl](5.3,2.255){i+2}
				
				\rput[bl](4.95,2.625){\fontsize{7 pt}{1 pt}\color {black} \selectfont i}
				\rput[bl](4.55,2.225){\fontsize{7 pt}{1 pt}\color {black} \selectfont i}
				\rput[bl](4.15,1.825){\fontsize{7 pt}{1 pt}\color {black} \selectfont i}
				\rput[bl](3.75,1.425){\fontsize{7 pt}{1 pt}\color {black} \selectfont i}
				\rput[bl](3.35,1.025){\fontsize{7 pt}{1 pt}\color {black} \selectfont i}
				\rput[bl](2.95,0.625){\fontsize{7 pt}{1 pt}\color {black} \selectfont i}
				
				\rput[bl](0.95,-1.325){\fontsize{7 pt}{1 pt}\color {black} \selectfont i}
				\rput[bl](0.55,-1.725){\fontsize{7 pt}{1 pt}\color {black} \selectfont i}
				\rput[bl](1.25,-1.325){\fontsize{6 pt}{1 pt}\color {black} \selectfont i+1}
				\rput[bl](0.85,-1.725){\fontsize{6 pt}{1 pt}\color {black} \selectfont i+1}
				\rput[bl](0.45,-2.125){\fontsize{6 pt}{1 pt}\color {black} \selectfont i+1}

				\rput[bl](4.85,2.225){\fontsize{6 pt}{1 pt}\color {black} \selectfont i+1}
				\rput[bl](4.45,1.825){\fontsize{6 pt}{1 pt}\color {black} \selectfont i+1}
				\rput[bl](4.05,1.425){\fontsize{6 pt}{1 pt}\color {black} \selectfont i+1}
				\rput[bl](3.65,1.025){\fontsize{6 pt}{1 pt}\color {black} \selectfont i+1}
				\rput[bl](3.25,0.625){\fontsize{6 pt}{1 pt}\color {black} \selectfont i+1}
				\end{pspicture}
			}
			\centering
			\caption{Covering the $B_2$ cells with partial-trades}
			\label{exp_tablediagonal}
		\end{figure}
		
		As the part $B$ is the union of two parts $B_1$ and $B_2$, we have covered all the cells in part $B$ by partial-trades of type $1A$ and $1C$ up to this point.
		\subsection{Covering $C_2$ and $D_2$ by partial-trades}
		\label{covC_2andD_2}
		In this section, we try to cover the rest of the cells in the four columns $t-3$ to $t$ by partial trades. According to the current placements of partial-trades, it can be concluded that for each row in the latin representation zero or two cells of part $C_2$ are not covered. Consider \begin{sf}$f(i,j)$\end{sf} as the number of rows with uncovered cells in columns $i$ and $j$ where $t-3 \leq i < j \leq t$. Doing some simple calculation, the value of $f$ can be computed as follows:
		\begin{enumerate}
			\item $s < 2r+4$:
			\begin{center}
				$f(t-3,t-1)=2r+2-s+\frac{s-r-2}{2}, f(t-3,t)=\frac{s-r-2}{2}$
				\\
				$f(t-2,t-1)=1,f(t-1,t)=\frac{s-r}{2}$
			\end{center}
			\item $s \geq 2r+4$:
			\begin{center}
				$f(t-3,t-1)=\frac{2r+3-(s-r)-1}{2}+s-(2r+3), f(t-3,t)=\frac{2r+4-(s-r)}{2}$
				\\
				$f(t-2,t-1)=1, f(t-1,t)=\frac{s-r}{2}$
			\end{center}
		\end{enumerate}
		Moreover, The value of the function $f$ for all other combinations of the last 4 columns that were not mentioned in the above equation equals 0.\\
		If $\frac{s-r-2}{2}$ is an odd number, the parity of $f(t-3,t-1)$, $f(t-3,t)$, $f(t-2,t-1)$ and $f(t-1,t)$  are even, odd, odd and even, respectively. If, however, $\frac{s-r-2}{2}$ is an even number, the parity of $f(t-3,t-1)$, $f(t-3,t)$, $f(t-2,t-1)$ and $f(t-1,t)$  are odd, even, odd and odd, respectively.
		\begin{figure}[H]
			\psscalebox{0.8 0.8} 
			{
				\begin{pspicture}(0,-2.11)(4.8,2.11)
				\definecolor{colour1}{rgb}{0.0,0.0,0.2}
				\psframe[linecolor=black, linewidth=0.04, dimen=outer](4.8,1.49)(0.0,-2.11)
				\psline[linecolor=black, linewidth=0.04](1.2,-2.11)(1.2,1.49)
				\psline[linecolor=black, linewidth=0.04](2.4,1.49)(2.4,-2.11)
				\psline[linecolor=black, linewidth=0.04](3.6,-2.11)(3.6,1.49)
				\psline[linecolor=black, linewidth=0.04](0.0,0.29)(4.8,0.29)
				\psline[linecolor=black, linewidth=0.04](4.8,-0.91)(0.0,-0.91)
				\psframe[linecolor=black, linewidth=0.04, fillstyle=solid,fillcolor=colour1, dimen=outer](4.8,1.49)(3.6,0.29)
				\psframe[linecolor=black, linewidth=0.04, fillstyle=solid,fillcolor=colour1, dimen=outer](2.4,1.49)(1.2,0.29)
				\psframe[linecolor=black, linewidth=0.04, fillstyle=solid,fillcolor=colour1, dimen=outer](2.4,0.29)(1.2,-0.91)
				\psframe[linecolor=black, linewidth=0.04, fillstyle=solid,fillcolor=colour1, dimen=outer](1.2,0.29)(0.0,-0.91)
				\psframe[linecolor=black, linewidth=0.04, fillstyle=solid,fillcolor=colour1, dimen=outer](2.4,-0.91)(1.2,-2.11)
				\psframe[linecolor=black, linewidth=0.04, fillstyle=solid,fillcolor=colour1, dimen=outer](3.6,-0.91)(2.4,-2.11)
				\rput[bl](4.0,1.69){\fontsize{10 pt}{3 pt}\selectfont t}
				\rput[bl](2.8,1.69){\fontsize{10 pt}{3 pt}\selectfont t-1}
				\rput[bl](1.6,1.69){\fontsize{10 pt}{3 pt}\selectfont t-2}
				\rput[bl](0.4,1.69){\fontsize{10 pt}{3 pt}\selectfont t-3}
				\end{pspicture}
			}
			\centering
			\caption{The position of covered cells where $\frac{s-r-2}{2}$ is even}
			\label{exp_t_3_t_1}
		\end{figure}
		Now, we classify the rows with two uncovered cells in part $C_{2}$ based on the columns that include the two uncovered cells. The cardinality of each class of rows has been specified by the function $f$. 
		Please recall when we identified partial-trades for part $B_{1}$, we used some of the cells in part $C_{2}$. Now we cover all the uncovered cells from part $C_{2}$ by partial-trades. For this purpose, we use cells of part $D_2$. Consider the problem in the following cases:
		\begin{enumerate}[I)]
			\item $s\neq r+2$:\\
			Define $l^{i,j}_{k}$ as the $k$-th row with uncovered cells in columns $i$ and $j$. If $\frac{(s-r-2)}{2}$ is odd, we consider two partial-trades of type 1B as follows:
			\begin{center}
				$pt_1=\lbrace(l^{t-2,t-1}_{1},t-2),(l^{t-2,t-1}_{1},t-1),(t-2,1),(t-1,1)\rbrace$\\
				$pt_2=\lbrace (l^{t-3,t}_{1},t-3),(l^{t-3,t}_{1},t),(t-3,1),(t,1) \rbrace$
			\end{center}  
			If $\frac{(s-r-2)}{2}$ is even, we take another partial-trade of type 1E. The new partial-trade is:
			\begin{equation*}
				\begin{split}
					& pt_3 = \lbrace (l^{t-3,t-1}_1, t-3), (l^{t-3,t-1}_1, t-1),\\  & (l^{t-3,t}_2, t-3),
					(l^{t-3,t}_2, t), (l^{t-1,t}_1, t-1), (l^{t-1,t}_1, t)\rbrace	
				\end{split}
			\end{equation*}
			
			
			\item $s=r+2$:\\
			In this case, we can use $(t-1)$-cells instead of $(t)$-cells in their corresponding partial-trades in rows $3$ and $4$ of part $C_2$. Then, we use the same method as the case when $s\neq r+2$.
		\end{enumerate}
		It is evident that after doing the above operations, the value of $f$ is even for all pair of 4 columns in $C_2$. Now start from row $1$ and pair each row with two uncovered cells with the next row with uncovered cells in the same columns as a partial-trade of type $1A$. We can observe that any two paired rows have distance at most $2$.\\
		To cover part $D_2$ by partial-trades, For each $i(1\leq i \leq r)$, we know that either both or none of the two cells $(i,t-3)$ and $(i,t)$ were covered by partial-trades. The same applies to the rows $t-2$ and $t-1$. Hence, we can easily partition the remaining cells of these paired rows by partial-trades of type $1A$.\\
		\subsection{Covering $C_1$ and $D_1$ by partial-trades}
		\label{sec::c1andd1}
		Up to this point, all cells have been covered by partial trades except for the cells in parts $C_1$ and $D_1$. Define $l = (t-4)-(s+1) + 1$. As $l$ is even, we can pair row $s+1$ with row $s+2$, row $s+3$ with row $s+4$, ..., and row $t-5$ with row $t-4$. Columns can be paired following the same procedure; i.e., column $s+1$ with column $s+2$, column $s+3$ with column $s+4$, ..., and column $t-5$ with column $t-4$. Now, for each pair of matched rows and matched columns with the same set of values in their cells, we use one $1B$ partial-trade - the exact placement of this partial-trade will be specified later in section \ref{tradec1d1}. Until now, in each pair of matched-columns there is exactly one row which their intersection used the $1B$ partial-trade. In $C_1$, by considering the partial-trade $1B$ in each matched columns, $s-1$ uncovered cells remains. If we consider the row of partial-trade $1B$ as $x$, we can match row $x+1$ with row $x+2$, row $x+3$ with row $x+4 $, $...$, and row $x+s-2$ with row $x+s-1$; therefore, each pair of rows form a $1A$ partial-trade. We use the same method to cover the cells of part $D_1$.
		
		\section{Expanding to trades}\label{sec::expand to trades}
		In this section, we expand the considered partial-trades to trades. First, we state some definitions and lemmas.
		\begin{definitiona}
			A \begin{sf}dual cell\end{sf} is a pair of cells in a diagonal in two consecutive rows of part $A$. We use the notation \begin{sf}$DC_{v,i}$\end{sf} to denote the dual cell consisting of cells in the diagonal $v$ and rows $i$ and $i+1$.
		\end{definitiona}
		\begin{definitiona} 
			The relation `$\leq$' on dual cells is as follows:
			\begin{equation*}
				DC_{v,i} \leq DC_{v',i'} \iff L(v) < L(v')\ \vee\ (v = v'\ \wedge\ (v - i)  \geq (v' - i')\pmod s)
			\end{equation*}
		\end{definitiona}
		This relation also defines a relation with which the set of dual cells will be a totally ordered set.
		\begin{definitiona}
			A dual cell $DC_{v,i}$ is \begin{sf}odd\end{sf} if $i$ is odd and it is \begin{sf}even\end{sf} if $i$ is even.
		\end{definitiona}
		\begin{figure}[H]
			\psscalebox{0.6 0.6}
			{
				\begin{pspicture}(0,-5.8)(10.4,5.8)
				\definecolor{colour0}{rgb}{0.8,0.8,0.8}
				\psframe[linecolor=black, linewidth=0.04, dimen=outer](4.8,5.8)(0.0,0.2)
				\psframe[linecolor=black, linewidth=0.04, dimen=outer](10.4,5.8)(5.6,0.2)
				\psframe[linecolor=black, linewidth=0.04, fillstyle=solid,fillcolor=colour0, dimen=outer](1.2,3.0)(0.8,2.6)
				\psframe[linecolor=black, linewidth=0.04, fillstyle=solid,fillcolor=colour0, dimen=outer](1.6,3.4)(1.2,3.0)
				\psframe[linecolor=black, linewidth=0.04, dimen=outer](2.0,3.8)(1.6,3.4)
				\psframe[linecolor=black, linewidth=0.04, dimen=outer](2.4,4.2)(2.0,3.8)
				\psframe[linecolor=black, linewidth=0.04, dimen=outer](6.8,3.8)(6.4,3.4)
				\psframe[linecolor=black, linewidth=0.04, dimen=outer](7.2,4.2)(6.8,3.8)
				\psframe[linecolor=black, linewidth=0.04, fillstyle=solid,fillcolor=colour0, dimen=outer](10.0,3.4)(9.6,3.0)
				\psframe[linecolor=black, linewidth=0.04, fillstyle=solid,fillcolor=colour0, dimen=outer](9.6,3.0)(9.2,2.6)
				\psline[linecolor=black, linewidth=0.03](0.4,0.2)(0.4,5.8)
				\psline[linecolor=black, linewidth=0.03](0.8,5.8)(0.8,0.2)
				\psline[linecolor=black, linewidth=0.03](1.2,0.2)(1.2,5.8)
				\psline[linecolor=black, linewidth=0.03](1.6,5.8)(1.6,0.2)
				\psline[linecolor=black, linewidth=0.03](2.0,0.2)(2.0,5.8)
				\psline[linecolor=black, linewidth=0.03](2.4,5.8)(2.4,0.2)
				\psline[linecolor=black, linewidth=0.03](0.0,4.2)(4.8,4.2)
				\psline[linecolor=black, linewidth=0.03](4.8,3.8)(0.0,3.8)
				\psline[linecolor=black, linewidth=0.03](0.0,3.4)(4.8,3.4)
				\psline[linecolor=black, linewidth=0.03](4.8,3.0)(0.0,3.0)
				\psline[linecolor=black, linewidth=0.03](0.0,2.6)(4.8,2.6)
				\psline[linecolor=black, linewidth=0.03](4.8,2.2)(0.0,2.2)
				\psline[linecolor=black, linewidth=0.03](5.6,4.2)(10.4,4.2)
				\psline[linecolor=black, linewidth=0.03](10.4,3.8)(5.6,3.8)
				\psline[linecolor=black, linewidth=0.03](5.6,3.4)(10.4,3.4)
				\psline[linecolor=black, linewidth=0.03](10.4,3.0)(5.6,3.0)
				\psline[linecolor=black, linewidth=0.03](5.6,2.6)(10.4,2.6)
				\psline[linecolor=black, linewidth=0.03](6.0,5.8)(6.0,0.2)
				\psline[linecolor=black, linewidth=0.03](6.4,0.2)(6.4,5.8)
				\psline[linecolor=black, linewidth=0.03](6.8,5.8)(6.8,0.2)
				\psline[linecolor=black, linewidth=0.03](7.2,0.2)(7.2,5.8)
				\psline[linecolor=black, linewidth=0.03](9.2,0.2)(9.2,5.8)
				\psline[linecolor=black, linewidth=0.03](9.6,5.8)(9.6,0.2)
				\psline[linecolor=black, linewidth=0.03](10.0,0.2)(10.0,5.8)
				\psframe[linecolor=black, linewidth=0.04, dimen=outer](7.6,-0.2)(2.8,-5.8)
				\psframe[linecolor=black, linewidth=0.04, dimen=outer](4.4,-2.6)(4.0,-3.0)
				\psframe[linecolor=black, linewidth=0.04, dimen=outer](4.8,-2.2)(4.4,-2.6)
				\psframe[linecolor=black, linewidth=0.04, fillstyle=solid,fillcolor=colour0, dimen=outer](7.6,-3.0)(7.2,-3.4)
				\psframe[linecolor=black, linewidth=0.04, fillstyle=solid,fillcolor=colour0, dimen=outer](7.2,-3.4)(6.8,-3.8)
				\psline[linecolor=black, linewidth=0.03](3.2,-5.8)(3.2,-0.2)
				\psline[linecolor=black, linewidth=0.03](3.6,-0.2)(3.6,-5.8)
				\psline[linecolor=black, linewidth=0.03](4.0,-5.8)(4.0,-0.2)
				\psline[linecolor=black, linewidth=0.03](4.4,-0.2)(4.4,-5.8)
				\psline[linecolor=black, linewidth=0.03](4.8,-5.8)(4.8,-0.2)
				\psline[linecolor=black, linewidth=0.03](7.2,-5.8)(7.2,-0.2)
				\psline[linecolor=black, linewidth=0.03](6.8,-0.2)(6.8,-5.8)
				\psline[linecolor=black, linewidth=0.03](7.6,-3.8)(2.8,-3.8)
				\psline[linecolor=black, linewidth=0.03](2.8,-3.4)(7.6,-3.4)
				\psline[linecolor=black, linewidth=0.03](7.6,-3.0)(2.8,-3.0)
				\psline[linecolor=black, linewidth=0.03](2.8,-2.6)(7.6,-2.6)
				\psline[linecolor=black, linewidth=0.03](7.6,-2.2)(2.8,-2.2)
				\psline[linecolor=black, linewidth=0.03](6.4,-5.8)(6.4,-0.2)
				\psline[linecolor=black, linewidth=0.03](2.8,-4.2)(7.6,-4.2)
				\end{pspicture}
			}
			\centering
			\caption{Examples of consecutive dual cells}
			\label{exp_nextdualcell}
		\end{figure}
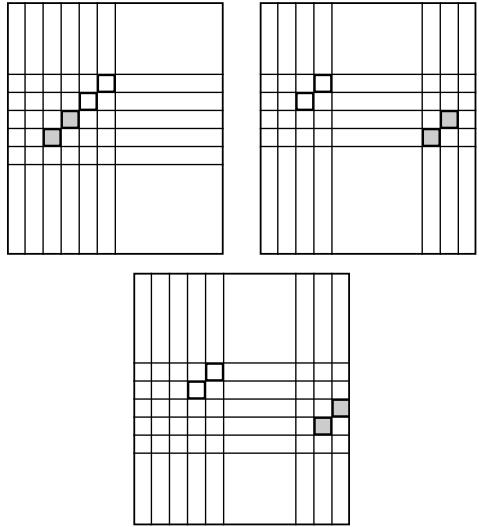
		
		In the previous steps, all cells in parts $B$, $C$, and $D$ were covered by partial-trades. In this step, we provide a method to expand them to trades. At first we expand partial-trades of parts $B_1$, $C_2$, and $D_2$. After that, we expand partial-trades of part $B_2$. Finally, we expand partial-trades of parts $C_1$ and $D_1$.
		\begin{lemma}\label{lemma:exis}
			There exists a number $x$ such that both diagonals with entries  $x$ and $x + r + 1 \pmod s$ have labels less than $s-2$.
		\end{lemma}
		\begin{proof}
			Consider all pairs $(i, i+r+1 \pmod s)$ with $1\leq i\leq s$. First note that there are at most 6 pairs $\{i, i+r+1\pmod s\}$ with either $i$ or $i+r+1$ being greater than or equal to $s-2$; i.e., being $s$, $s-1$, or $s-2$. Moreover, according to Lemma \ref{rslimit} we have $s \geq 18$, so by the pigeonhole theorem there is at least one $x$ that satisfies the lemma.
		\end{proof}
		
		\subsection{Expanding the partial trades for part $B_1$ and $C_1$ to trades}
		In order to avoid cluttering, all the numbers in this paragraph are modulo $s$. Consider an integer $x$ that satisfies the condition lemma \ref{lemma:exis}. In order to expand partial-trades of parts $B_1$ and $C_2$, we consider the diagonals $x+r+1$ and $x$ simultaneously. Then, for expanding partial-trades consisting of even cells in $B_1$, we expand partial-trades consisting of 2-cells with the first cells of diagonals $x+r+1$ and $x$. In general we expand partial-trades consisting of $2i$-cells with $(2i-1)$th cells of diagonals $x+r+1$ and $x$. Note that if we reach the cells in column $2$ of the diagonals, before using these cells, we consider their next diagonal, according to the label ordering, and without using the two first columns of $A$, we continue the procedure of expanding partial-trades by first cells of these two new diagonals.
		
		In section \ref{sec::partial-trades}, We have used two partial-trades of type $1B$ and at most one partial-trade of type $1E$, to change the parity of all the values of $f$. To expand partial-trade $1E$ we use the second column of part $A$, and to expand partial-trades $1B$ we use the first and second column of part $A$. Given that, expanded cells from part $A$ are uniquely determined and it is easy to check that all the obtained trades do not overlap.
		\begin{lemma}
			In the above procedure, there is no repeated diagonal.
		\end{lemma}
		\begin{proof}
			First note that as it was mentioned in the labeling procedure, the distance between consecutive diagonals is either $r-5$ or $r-3$.
			Without loss of generality, we can assume that the repeated diagonal is either $x$ or $x+r+1$; i.e., we start from one of the diagonals with label $x$ or $x+r+1$, we keep moving to the next diagonal, based on the label ordering, until we reach the other diagonal. We define $y_1$ and $y_2$ as the number of pairs of consecutive diagonals with distance $r-5$ and $r-3$ in the mentioned process, respectively. First, note that the diagonals used in the process of reaching from a diagonal $x$ or $x+r+1$ to the other one do not cover the entire rows of the table and so $y_1=0$ and $y_2(r-3) < s$. Therefore, this is sufficient for us to consider the two following cases to conclude the proof.
			\begin{itemize}
				\item From $x$ to $x+r+1$: in this case, $y_2(r-3) = s-r-1\pmod s$.
				We can conclude $y_2(r-3) < s$ and $y_2(r-3) = s-r-1\pmod s$ that $y_2(r-3) = s-r-1$, but the two sides of equation have different parity.
				\item From $r+2$ to 1: in this case, $y_2(r-3) = r+1 \pmod s$.
				According to Lemma \ref{rslimit}, $r > 5$, so the above equation is not possible.
			\end{itemize}
		\end{proof}
		\\For expanding partial-trades consisting odd cells of part $B_1$ except $s$-cells, we consider first unused even dual cell, $DC_{v,i}$, and expand the partial-trade in row $i$ and $i+1$ that includes odd cells in part $B_1$. We do the same for next dual cells and continue until all of these partial-trades are expanded.\\
		Now, we expand the partial-trades that include $s$-cells in $B_1$. To expand these partial-trades, we start from the first unused odd dual cell, $DC_{v,i}$. Expand the partial-trade consisting $s$-cells in rows $i$ and $i+1 \pmod s$ (if exist) and go to next dual cell in part $A$ until we reach the last row. Then, from next 2 dual cells that are in rows ($s$,$1$) and rows ($2$,$3$), we skip the cell in part $A$ in the row $s$. Now, consider the cells in the first and the second rows as a dual cell and then consider s-cells in the same row of this dual cell as a trade and go to next dual cell and identify trades until all s-cells in part $B_1$ are covered with trades.
		
		We know that rows of partial-trades of this part have distance at most 2. Now, according to the procedure of expansion of partial-trades of part $B_1$, there are some unused cells in diagonals of part $A$. We expand partial-trades of part $C_2$ whose rows have distance 2 using the cells in part $A$.
		For expanding partial-trades of part $C_1$ whose rows have distance 1, we can start from the first unused dual cell and keep moving to the next dual cell until we cover all those partial-trades. Note that it may happen that some cells will be unused in this procedure, and as we have mentioned earlier that in each row of part $C_2$ there are 0 or 2 cells which are not covered by partial-trades of part $B_1$. As such, in this part of algorithm, we are passing through at most $s$ cells except cells in 3 first diagonals to cover each row of the table.
		
		According to cases for $s$ mentioned in section \ref{covC_2andD_2} and identified partial-trades, all the partial-trades that has an intersection with the part $B_1$ are expanded to trades by the above procedure.
		\subsection{Expanding the partial trades for part $B_2$ to trades}
		
		For expanding partial-trades of part $B_2$, we partition partial-trades into two groups:
		\begin{enumerate}
			\item partial-trades with odd top row:\\
			Consider $a=\frac{s-r-2}{2}$, then the number of partial-trades in rows $2k-1$ and $2k$ where $(1\leq k \leq \frac{s-1}{2})$ is equal to $max\{0,a-k+1\}$.\\
			\item partial-trades with even top row:\\
			Consider $a=\frac{s-r-2}{2}$, then the number of partial-trades in rows $2k$ and $2k+1$ where $(1\leq k \leq \frac{s-1}{2})$ is equal to $min\{k-1,a\}$.\\
		\end{enumerate}
		
		Notice that the union of $DC_{v,i}$ with partial-trades of type $1A$ whose cells are in rows $i$ and $i+1$, form a trade of type $1A$. Given that, the following paragraph shows the procedure of expanding partial-trades of $B_2$. 
		\newline
		Define $\psi(x,i)$, where $x$ is a dual cell $DC_{v,r}$ and $i$ is a non-negative integer, as a function returning the smallest unused dual cell which is greater than or equal to $x$ and it has a form $DC_{., r+ i \pmod s}$. Note that $\psi$ is not a predetermined function and its value depends on the current usage of dual cells.
		
		Let $T$ be a sequence of dual cells with length $a(r+1)$ where $T(1)$ is the minimum unused dual cell whose cells are in rows 1 and 2 and for T(i) ($1 < i \leq a(r+1)$):
		\begin{equation*}
			T_{i} = \psi(T_1, 2(i-1) + \lfloor\frac{(i + (\frac{s-3}{2}-a-1)}{r}\rfloor )
		\end{equation*}
		With a simple calculation we can find out that the marked dual cells identified by the above procedure, are sufficient to expand all the partial-trades.
		
		\subsection{Expanding the partial trades for part $C_1$, $D_1$, and $D_2$ to trades}\label{tradec1d1}

		As mentioned in section \ref{sec::c1andd1}, we partition part $C_1$ into $\frac{t-s-4}{2}$ column groups where the $(i)$th group ($1\leq i\leq \frac{t-s-4}{2}$) includes all the cells with entry $s+2i-1$ and $s+2i$. We apply the same grouping method to $D_1$ so we get $\frac{t-s-4}{2}$ row groups. Then, we choose the first diagonal according to labels in part $A$ whose cells are not used and partition it into $\frac{t-s-4}{2}$ dual cells such that dual cell $1\leq i\leq \frac{t-s-4}{2}$ consists $2i-1$th and $2i$th cells from the top of the diagonal. For each $i$, we consider two cells from part $C_1$ from the $i$th column group and two cells from part $D_1$ from the $i$th row group such that the union of the four cells and the $i$th dual cell in the mentioned partition form a trade of type $1B$.
		Trivially, each of these partial-trade groups in part $D_1$ can be expanded only by one diagonal in part $A$.
		
		For part $C_1$ we start from first unused dual cell in part $A$ whose union with one partial-trade of column group $1$ in part $C_1$ forms a trade. In an iterative procedure, We identify the next trades by considering the union of the next dual cell and the corresponding partial-trade in part $C_1$  until we reach the dual cell whose first cell is in the same row with the first matched pair. Then, we consider second and third cells of the union of the next two dual cells as a dual cell and expand the corresponding partial trade to trade in part $C_1$. Then, we continue the procedure of identifying trades as before until we can cover all cells in column group $1$ in $C_1$. We do the same for other column groups until we cover all cells in part $C_1$.
		
		In order to expand subtrades of part $D_2$, we can easily verify that with each 2 diagonals in part $A$ we can expand subtrades of part $D_2$ to trades of type $1C$.
		
		\section{Main Result}
		Now we explore the sufficient conditions under which the proposed decomposition algorithm will terminate. For this purpose, we count the number of uncovered cells in part $A$. This is simpler if the process of algorithm is divided into small steps as follows. Steps are listed in the order of their execution:
		
		\begin{enumerate}
			\item First note that the number of cells in 3 first columns and two last columns of part A are 5s, so from now we do not count them twice. Please note that throughout the algorithm description we avoid using the cells in those mentioned columns, so we can use them for specific purposes.
			\item Even cells of $B_1$: For each diagonal of part $A$ used in this step half of their cells are used in expanding partial-trades.
			Moreover, we use two sequence of consecutive diagonals each sequence crosses one row at most once, since we just have two partial trades for even cells of part $B_1$ in each row. For each sequence of consecutive diagonals, we have at most $\frac{s-1}{2}$ unused cells, totaling $s-1$ number of unused cells.

			Also, the first cell in these two consecutive diagonals is the first cell of diagonals in part $A$. As such, when we encounter them in the process of expanding partial-trades to trades, we have to ignore at most $r-3$ cells to the next unused diagonal. Please note we might ignore a fewer number of cells; for example, when we are covering parts $D_1$, we ignore at most 2 cells to reach the first unused dual cell.
			\item Odd cells of $B_1$: At the beginning of this step, we ignore at most one cell in order to reach the smallest unused even dual cell.
			Throughout expanding partial trades covering $s$ cells of part $B_1$, at most $r$ cells are ignored to reach the next unused dual cell at the same pair of rows of these partial trades.
			\item $C_2$: From Section \ref{covC_2andD_2}  it can be concluded that in each row of the latin representation, zero or two cells of part $C_2$ remain uncovered. All the partial-trades in this part were expanded by trades of type $1A$. Moreover, we require less than $s$ cells to cover all partial-trades of part $C_2$, since there are at most 1 partial-trade for each row of part $C_2$. As such, We ignore at most $s$ cells, because there are rows in part $C_2$ whose partial trades have been covered in the procedure of covering partial-trades of part $B_1$.
			\item $B_2$: At the beginning of this step, we keep moving to the next dual cell until we reach the smallest unused dual cell in the first two rows of the table, so we ignore at most $s$ cells to start this step. Moreover, we ignore one cell whenever we pass through all rows of the table. As such, at most $s-r$ cells are ignored after this step initiated.
			Hence, at most $2s-r$ cells in total was ignored in this step.
			\item $C_1$: At the beginning of this step, we keep moving to next dual cell until we reach the first unused diagonal. Hence, we ignore at most $r$ cells to reach the next diagonal. However, there are at most five diagonals whose cells in two first columns of part $A$ are used before. As such, we ignore at most $r+5r = 6r$ cells to initiate this step. 
			During the step, at most $\frac{s-r}{2}$ cells were ignored whilst we formed the type $1B$ trades in parts $C_1$ and $D_1$. Hence, at most $\frac{s+11r}{2}$ cells overall was ignored in this step.
			\item $D_1$ and $D_2$: At the beginning of this step, we keep moving to the next dual cell until we reach the first unused diagonal. As a result, we ignore at most $r$ cells to initiate this step.\\
		\end{enumerate}
		
		Now we compute the number of ignored cells counted in the 7 steps above:
		\begin{equation*}\label{eq:ignored}
			\begin{split}
				&5s + (s+1+r-3) + (1+r) + s + (2s - r) + \frac{s+11r}{2} + r\\
				&\leq 9.5s + 7.5r - 2 \leq 9(s+r)
			\end{split}
		\end{equation*}
		For each trade, number of cells used outside of part $A$ are twice the number of cells used in that trade inside part $A$. In order for the algorithm to have enough cells to cover all cells in parts $B,C$, and $D$, the number of cells outside of part $A$ should be at most twice the number of cells inside part $A$ excluding the cells ignored throughout the algorithm:
		\begin{equation*}
			s(t-r) + r(t-s) \le 2(rs - 9(r+s)) \Rightarrow 
		\end{equation*}
		Which can be rewritten as:
		\begin{equation*}\label{eq:final}
			t(r+s) \le 4rs - 18(r+s) \Rightarrow t+18 \le \frac{4rs}{r+s}
		\end{equation*}
		According to the above inequality, the sufficient condition for validity of our algorithm is stronger than the necessary conditions for the problem only by a constant number, $18$, and condition $t \neq s+2 $. The following Remark summarizes our results.
		
		\begin{remark}
			\label{main_th}
			There exists a decomposition of $K_{5r,5s,5t}$ into 5-cycles if the following conditions hold:
			\begin{itemize}
				\item $t+18 \leq \frac{4rs}{r+s}$
				\item $t \neq s+2$
				\item $r$, $s$ and $t$ have the same parity.
			\end{itemize}
		\end{remark}
		
		We use our method for decomposing $K_{15,17,21}$ into 5-cycles and triangles in appendix \ref{sec::appendix}.
		\section{Conclusions and Open Problems}
		We proposed a method to solve one of the remaining cases of Mahmoodian and Mirzakhani \cite{MR1366852} conjecture. This case was stated in Theorem \ref{main_th}, although there exists a constant gap in the inequality introduced by our method and the inequality of the conjecture. The remaining case to prove this conjecture is when all parts have odd sizes and at least one of them is not a multiple of $5$.
		
		\section*{Appendices}\label{sec::appendix}
		
		\begin{figure}[H]
			\centering
			\begin{tikzpicture}[scale=.432]
			\begin{scope}
			\tiny
			\draw (0,0) -- (31.5,0);\draw (0,1.5) -- (31.5,1.5);\draw (0,3) -- (31.5,3);\draw (0,4.5) -- (31.5,4.5);\draw (0,6) -- (31.5,6);\draw (0,7.5) -- (31.5,7.5);\draw (0,9) -- (31.5,9);\draw (0,10.5) -- (31.5,10.5);\draw (0,12) -- (31.5,12);\draw (0,13.5) -- (31.5,13.5);\draw (0,15) -- (31.5,15);\draw (0,16.5) -- (31.5,16.5);\draw (0,18) -- (31.5,18);\draw (0,19.5) -- (31.5,19.5);\draw (0,21) -- (31.5,21);\draw (0,22.5) -- (31.5,22.5);\draw (0,24) -- (31.5,24);\draw (0,25.5) -- (31.5,25.5);
			\draw (0,0) -- (22.5,0);\draw (0,-1.5) -- (22.5,-1.5);\draw (0,-3) -- (22.5,-3);\draw (0,-4.5) -- (22.5,-4.5);\draw (0,-6) -- (22.5,-6);
			\draw (0,-6) -- (0,25.5);\draw (1.5,-6) -- (1.5,25.5);\draw (3,-6) -- (3,25.5);\draw (4.5,-6) -- (4.5,25.5);\draw (6,-6) -- (6,25.5);\draw (7.5,-6) -- (7.5,25.5);\draw (9,-6) -- (9,25.5);\draw (10.5,-6) -- (10.5,25.5);\draw (12,-6) -- (12,25.5);\draw (13.5,-6) -- (13.5,25.5);\draw (15,-6) -- (15,25.5);\draw (16.5,-6) -- (16.5,25.5);\draw (18,-6) -- (18,25.5);\draw (19.5,-6) -- (19.5,25.5);\draw (21,-6) -- (21,25.5);\draw (22.5,-6) -- (22.5,25.5);
			\draw (22.5,0) -- (22.5,25.5);\draw (24,0) -- (24,25.5);\draw (25.5,0) -- (25.5,25.5);\draw (27,0) -- (27,25.5);\draw (28.5,0) -- (28.5,25.5);\draw (30,0) -- (30,25.5);\draw (31.5,0) -- (31.5,25.5);
			\draw[very thick] (0,0) rectangle (22.5,25.5);
			\draw[very thick] (0,0) rectangle (22.5,-6);
			\draw[very thick] (22.5,0) rectangle (25.5,25.5);
			\draw[very thick] (25.5,0) rectangle (31.5,25.5);
			\node at (0.75,25.15) {1};\node at (0.75,23.65) {2};\node at (0.75,22.15) {3};\node at (0.75,20.65) {4};\node at (0.75,19.15) {5};\node at (0.75,17.65) {6};\node at (0.75,16.15) {7};\node at (0.75,14.65) {8};\node at (0.75,13.15) {9};\node at (0.75,11.65) {10};\node at (0.75,10.15) {11};\node at (0.75,8.65) {12};\node at (0.75,7.15) {13};\node at (0.75,5.65) {14};\node at (0.75,4.15) {15};\node at (0.75,2.65) {16};\node at (0.75,1.15) {17};
			\node at (2.25,25.15) {2};\node at (2.25,23.65) {3};\node at (2.25,22.15) {4};\node at (2.25,20.65) {5};\node at (2.25,19.15) {6};\node at (2.25,17.65) {7};\node at (2.25,16.15) {8};\node at (2.25,14.65) {9};\node at (2.25,13.15) {10};\node at (2.25,11.65) {11};\node at (2.25,10.15) {12};\node at (2.25,8.65) {13};\node at (2.25,7.15) {14};\node at (2.25,5.65) {15};\node at (2.25,4.15) {16};\node at (2.25,2.65) {17};\node at (2.25,1.15) {1};
			\node at (3.75,25.15) {3};\node at (3.75,23.65) {4};\node at (3.75,22.15) {5};\node at (3.75,20.65) {6};\node at (3.75,19.15) {7};\node at (3.75,17.65) {8};\node at (3.75,16.15) {9};\node at (3.75,14.65) {10};\node at (3.75,13.15) {11};\node at (3.75,11.65) {12};\node at (3.75,10.15) {13};\node at (3.75,8.65) {14};\node at (3.75,7.15) {15};\node at (3.75,5.65) {16};\node at (3.75,4.15) {17};\node at (3.75,2.65) {1};\node at (3.75,1.15) {2};
			\node at (5.25,25.15) {4};\node at (5.25,23.65) {5};\node at (5.25,22.15) {6};\node at (5.25,20.65) {7};\node at (5.25,19.15) {8};\node at (5.25,17.65) {9};\node at (5.25,16.15) {10};\node at (5.25,14.65) {11};\node at (5.25,13.15) {12};\node at (5.25,11.65) {13};\node at (5.25,10.15) {14};\node at (5.25,8.65) {15};\node at (5.25,7.15) {16};\node at (5.25,5.65) {17};\node at (5.25,4.15) {1};\node at (5.25,2.65) {2};\node at (5.25,1.15) {3};
			\node at (6.75,25.15) {5};\node at (6.75,23.65) {6};\node at (6.75,22.15) {7};\node at (6.75,20.65) {8};\node at (6.75,19.15) {9};\node at (6.75,17.65) {10};\node at (6.75,16.15) {11};\node at (6.75,14.65) {12};\node at (6.75,13.15) {13};\node at (6.75,11.65) {14};\node at (6.75,10.15) {15};\node at (6.75,8.65) {16};\node at (6.75,7.15) {17};\node at (6.75,5.65) {1};\node at (6.75,4.15) {2};\node at (6.75,2.65) {3};\node at (6.75,1.15) {4};
			\node at (8.25,25.15) {6};\node at (8.25,23.65) {7};\node at (8.25,22.15) {8};\node at (8.25,20.65) {9};\node at (8.25,19.15) {10};\node at (8.25,17.65) {11};\node at (8.25,16.15) {12};\node at (8.25,14.65) {13};\node at (8.25,13.15) {14};\node at (8.25,11.65) {15};\node at (8.25,10.15) {16};\node at (8.25,8.65) {17};\node at (8.25,7.15) {1};\node at (8.25,5.65) {2};\node at (8.25,4.15) {3};\node at (8.25,2.65) {4};\node at (8.25,1.15) {5};
			\node at (9.75,25.15) {7};\node at (9.75,23.65) {8};\node at (9.75,22.15) {9};\node at (9.75,20.65) {10};\node at (9.75,19.15) {11};\node at (9.75,17.65) {12};\node at (9.75,16.15) {13};\node at (9.75,14.65) {14};\node at (9.75,13.15) {15};\node at (9.75,11.65) {16};\node at (9.75,10.15) {17};\node at (9.75,8.65) {1};\node at (9.75,7.15) {2};\node at (9.75,5.65) {3};\node at (9.75,4.15) {4};\node at (9.75,2.65) {5};\node at (9.75,1.15) {6};
			\node at (11.25,25.15) {8};\node at (11.25,23.65) {9};\node at (11.25,22.15) {10};\node at (11.25,20.65) {11};\node at (11.25,19.15) {12};\node at (11.25,17.65) {13};\node at (11.25,16.15) {14};\node at (11.25,14.65) {15};\node at (11.25,13.15) {16};\node at (11.25,11.65) {17};\node at (11.25,10.15) {1};\node at (11.25,8.65) {2};\node at (11.25,7.15) {3};\node at (11.25,5.65) {4};\node at (11.25,4.15) {5};\node at (11.25,2.65) {6};\node at (11.25,1.15) {7};
			\node at (12.75,25.15) {9};\node at (12.75,23.65) {10};\node at (12.75,22.15) {11};\node at (12.75,20.65) {12};\node at (12.75,19.15) {13};\node at (12.75,17.65) {14};\node at (12.75,16.15) {15};\node at (12.75,14.65) {16};\node at (12.75,13.15) {17};\node at (12.75,11.65) {1};\node at (12.75,10.15) {2};\node at (12.75,8.65) {3};\node at (12.75,7.15) {4};\node at (12.75,5.65) {5};\node at (12.75,4.15) {6};\node at (12.75,2.65) {7};\node at (12.75,1.15) {8};
			\node at (14.25,25.15) {10};\node at (14.25,23.65) {11};\node at (14.25,22.15) {12};\node at (14.25,20.65) {13};\node at (14.25,19.15) {14};\node at (14.25,17.65) {15};\node at (14.25,16.15) {16};\node at (14.25,14.65) {17};\node at (14.25,13.15) {1};\node at (14.25,11.65) {2};\node at (14.25,10.15) {3};\node at (14.25,8.65) {4};\node at (14.25,7.15) {5};\node at (14.25,5.65) {6};\node at (14.25,4.15) {7};\node at (14.25,2.65) {8};\node at (14.25,1.15) {9};
			\node at (15.75,25.15) {11};\node at (15.75,23.65) {12};\node at (15.75,22.15) {13};\node at (15.75,20.65) {14};\node at (15.75,19.15) {15};\node at (15.75,17.65) {16};\node at (15.75,16.15) {17};\node at (15.75,14.65) {1};\node at (15.75,13.15) {2};\node at (15.75,11.65) {3};\node at (15.75,10.15) {4};\node at (15.75,8.65) {5};\node at (15.75,7.15) {6};\node at (15.75,5.65) {7};\node at (15.75,4.15) {8};\node at (15.75,2.65) {9};\node at (15.75,1.15) {10};
			\node at (17.25,25.15) {12};\node at (17.25,23.65) {13};\node at (17.25,22.15) {14};\node at (17.25,20.65) {15};\node at (17.25,19.15) {16};\node at (17.25,17.65) {17};\node at (17.25,16.15) {1};\node at (17.25,14.65) {2};\node at (17.25,13.15) {3};\node at (17.25,11.65) {4};\node at (17.25,10.15) {5};\node at (17.25,8.65) {6};\node at (17.25,7.15) {7};\node at (17.25,5.65) {8};\node at (17.25,4.15) {9};\node at (17.25,2.65) {10};\node at (17.25,1.15) {11};
			\node at (18.75,25.15) {13};\node at (18.75,23.65) {14};\node at (18.75,22.15) {15};\node at (18.75,20.65) {16};\node at (18.75,19.15) {17};\node at (18.75,17.65) {1};\node at (18.75,16.15) {2};\node at (18.75,14.65) {3};\node at (18.75,13.15) {4};\node at (18.75,11.65) {5};\node at (18.75,10.15) {6};\node at (18.75,8.65) {7};\node at (18.75,7.15) {8};\node at (18.75,5.65) {9};\node at (18.75,4.15) {10};\node at (18.75,2.65) {11};\node at (18.75,1.15) {12};
			\node at (20.25,25.15) {14};\node at (20.25,23.65) {15};\node at (20.25,22.15) {16};\node at (20.25,20.65) {17};\node at (20.25,19.15) {1};\node at (20.25,17.65) {2};\node at (20.25,16.15) {3};\node at (20.25,14.65) {4};\node at (20.25,13.15) {5};\node at (20.25,11.65) {6};\node at (20.25,10.15) {7};\node at (20.25,8.65) {8};\node at (20.25,7.15) {9};\node at (20.25,5.65) {10};\node at (20.25,4.15) {11};\node at (20.25,2.65) {12};\node at (20.25,1.15) {13};
			\node at (21.75,25.15) {15};\node at (21.75,23.65) {16};\node at (21.75,22.15) {17};\node at (21.75,20.65) {1};\node at (21.75,19.15) {2};\node at (21.75,17.65) {3};\node at (21.75,16.15) {4};\node at (21.75,14.65) {5};\node at (21.75,13.15) {6};\node at (21.75,11.65) {7};\node at (21.75,10.15) {8};\node at (21.75,8.65) {9};\node at (21.75,7.15) {10};\node at (21.75,5.65) {11};\node at (21.75,4.15) {12};\node at (21.75,2.65) {13};\node at (21.75,1.15) {14};
			\node at (23.25,25.15) {16};\node at (23.25,23.65) {17};\node at (23.25,22.15) {1};\node at (23.25,20.65) {2};\node at (23.25,19.15) {3};\node at (23.25,17.65) {4};\node at (23.25,16.15) {5};\node at (23.25,14.65) {6};\node at (23.25,13.15) {7};\node at (23.25,11.65) {8};\node at (23.25,10.15) {9};\node at (23.25,8.65) {10};\node at (23.25,7.15) {11};\node at (23.25,5.65) {12};\node at (23.25,4.15) {13};\node at (23.25,2.65) {14};\node at (23.25,1.15) {15};
			\node at (24.75,25.15) {17};\node at (24.75,23.65) {1};\node at (24.75,22.15) {2};\node at (24.75,20.65) {3};\node at (24.75,19.15) {4};\node at (24.75,17.65) {5};\node at (24.75,16.15) {6};\node at (24.75,14.65) {7};\node at (24.75,13.15) {8};\node at (24.75,11.65) {9};\node at (24.75,10.15) {10};\node at (24.75,8.65) {11};\node at (24.75,7.15) {12};\node at (24.75,5.65) {13};\node at (24.75,4.15) {14};\node at (24.75,2.65) {15};\node at (24.75,1.15) {16};
			\node at (26.25,25.15) {18};\node at (27.75,25.15) {19};\node at (29.25,25.15) {20};\node at (30.75,25.15) {21};
			\node at (26.25,23.65) {18};\node at (27.75,23.65) {19};\node at (29.25,23.65) {20};\node at (30.75,23.65) {21};
			\node at (26.25,22.15) {18};\node at (27.75,22.15) {19};\node at (29.25,22.15) {20};\node at (30.75,22.15) {21};
			\node at (26.25,20.65) {18};\node at (27.75,20.65) {19};\node at (29.25,20.65) {20};\node at (30.75,20.65) {21};
			\node at (26.25,19.15) {18};\node at (27.75,19.15) {19};\node at (29.25,19.15) {20};\node at (30.75,19.15) {21};
			\node at (26.25,17.65) {18};\node at (27.75,17.65) {19};\node at (29.25,17.65) {20};\node at (30.75,17.65) {21};
			\node at (26.25,16.15) {18};\node at (27.75,16.15) {19};\node at (29.25,16.15) {20};\node at (30.75,16.15) {21};
			\node at (26.25,14.65) {18};\node at (27.75,14.65) {19};\node at (29.25,14.65) {20};\node at (30.75,14.65) {21};
			\node at (26.25,13.15) {18};\node at (27.75,13.15) {19};\node at (29.25,13.15) {20};\node at (30.75,13.15) {21};
			\node at (26.25,11.65) {18};\node at (27.75,11.65) {19};\node at (29.25,11.65) {20};\node at (30.75,11.65) {21};
			\node at (26.25,10.15) {18};\node at (27.75,10.15) {19};\node at (29.25,10.15) {20};\node at (30.75,10.15) {21};
			\node at (26.25,8.65) {18};\node at (27.75,8.65) {19};\node at (29.25,8.65) {20};\node at (30.75,8.65) {21};
			\node at (26.25,7.15) {18};\node at (27.75,7.15) {19};\node at (29.25,7.15) {20};\node at (30.75,7.15) {21};
			\node at (26.25,5.65) {18};\node at (27.75,5.65) {19};\node at (29.25,5.65) {20};\node at (30.75,5.65) {21};
			\node at (26.25,4.15) {18};\node at (27.75,4.15) {19};\node at (29.25,4.15) {20};\node at (30.75,4.15) {21};
			\node at (26.25,2.65) {18};\node at (27.75,2.65) {19};\node at (29.25,2.65) {20};\node at (30.75,2.65) {21};
			\node at (26.25,1.15) {18};\node at (27.75,1.15) {19};\node at (29.25,1.15) {20};\node at (30.75,1.15) {21};
			\node at (0.75,-0.35) {18};\node at (0.75,-1.85) {19};\node at (0.75,-3.35) {20};\node at (0.75,-4.85) {21};
			\node at (2.25,-0.35) {18};\node at (2.25,-1.85) {19};\node at (2.25,-3.35) {20};\node at (2.25,-4.85) {21};
			\node at (3.75,-0.35) {18};\node at (3.75,-1.85) {19};\node at (3.75,-3.35) {20};\node at (3.75,-4.85) {21};
			\node at (5.25,-0.35) {18};\node at (5.25,-1.85) {19};\node at (5.25,-3.35) {20};\node at (5.25,-4.85) {21};
			\node at (6.75,-0.35) {18};\node at (6.75,-1.85) {19};\node at (6.75,-3.35) {20};\node at (6.75,-4.85) {21};
			\node at (8.25,-0.35) {18};\node at (8.25,-1.85) {19};\node at (8.25,-3.35) {20};\node at (8.25,-4.85) {21};
			\node at (9.75,-0.35) {18};\node at (9.75,-1.85) {19};\node at (9.75,-3.35) {20};\node at (9.75,-4.85) {21};
			\node at (11.25,-0.35) {18};\node at (11.25,-1.85) {19};\node at (11.25,-3.35) {20};\node at (11.25,-4.85) {21};
			\node at (12.75,-0.35) {18};\node at (12.75,-1.85) {19};\node at (12.75,-3.35) {20};\node at (12.75,-4.85) {21};
			\node at (14.25,-0.35) {18};\node at (14.25,-1.85) {19};\node at (14.25,-3.35) {20};\node at (14.25,-4.85) {21};
			\node at (15.75,-0.35) {18};\node at (15.75,-1.85) {19};\node at (15.75,-3.35) {20};\node at (15.75,-4.85) {21};
			\node at (17.25,-0.35) {18};\node at (17.25,-1.85) {19};\node at (17.25,-3.35) {20};\node at (17.25,-4.85) {21};
			\node at (18.75,-0.35) {18};\node at (18.75,-1.85) {19};\node at (18.75,-3.35) {20};\node at (18.75,-4.85) {21};
			\node at (20.25,-0.35) {18};\node at (20.25,-1.85) {19};\node at (20.25,-3.35) {20};\node at (20.25,-4.85) {21};
			\node at (21.75,-0.35) {18};\node at (21.75,-1.85) {19};\node at (21.75,-3.35) {20};\node at (21.75,-4.85) {21};
			\node at (2.25,24.35){\color{green}*23};\node at (9.75,24.35){\color{yellow}*17};\node at (12.75,24.35){\color{purple}*30};\node at (18.75,24.35){\color{orange}*8};\node at (20.25,24.35){\color{blue}*34};\node at (23.25,24.35){\color{orange}*8};\node at (24.75,24.35){\color{yellow}*17};\node at (26.25,24.35){\color{yellow}*17};\node at (27.75,24.35){\color{green}*24};\node at (29.25,24.35){\color{green}*24};\node at (30.75,24.35){\color{orange}*8};\node at (0.75,22.85){\color{green}*24};\node at (2.25,22.85){\color{green}*41};\node at (8.25,22.85){\color{yellow}*17};\node at (11.25,22.85){\color{purple}*30};\node at (15.75,22.85){\color{cyan}*9};\node at (18.75,22.85){\color{blue}*35};\node at (23.25,22.85){\color{yellow}*17};\node at (24.75,22.85){\color{cyan}*9};\node at (26.25,22.85){\color{yellow}*17};\node at (27.75,22.85){\color{cyan}*9};\node at (29.25,22.85){\color{green}*41};\node at (30.75,22.85){\color{green}*41};\node at (2.25,21.35){\color{green}*25};\node at (9.75,21.35){\color{purple}*31};\node at (14.25,21.35){\color{cyan}*9};\node at (17.25,21.35){\color{blue}*35};\node at (21.75,21.35){\color{orange}*1};\node at (23.25,21.35){\color{cyan}*9};\node at (24.75,21.35){\color{orange}*1};\node at (26.25,21.35){\color{green}*25};\node at (27.75,21.35){\color{cyan}*9};\node at (29.25,21.35){\color{orange}*1};\node at (30.75,21.35){\color{green}*25};\node at (0.75,19.85){\color{green}*25};\node at (2.25,19.85){\color{green}*41};\node at (8.25,19.85){\color{purple}*31};\node at (12.75,19.85){\color{cyan}*10};\node at (15.75,19.85){\color{blue}*36};\node at (21.75,19.85){\color{orange}*1};\node at (23.25,19.85){\color{orange}*1};\node at (24.75,19.85){\color{cyan}*10};\node at (26.25,19.85){\color{green}*41};\node at (27.75,19.85){\color{cyan}*10};\node at (29.25,19.85){\color{orange}*1};\node at (30.75,19.85){\color{green}*41};\node at (2.25,18.35){\color{green}*41};\node at (6.75,18.35){\color{purple}*32};\node at (11.25,18.35){\color{cyan}*10};\node at (14.25,18.35){\color{blue}*36};\node at (18.75,18.35){\color{orange}*2};\node at (23.25,18.35){\color{cyan}*10};\node at (24.75,18.35){\color{orange}*2};\node at (26.25,18.35){\color{green}*41};\node at (27.75,18.35){\color{cyan}*10};\node at (29.25,18.35){\color{green}*41};\node at (30.75,18.35){\color{orange}*2};\node at (5.25,16.85){\color{purple}*32};\node at (9.75,16.85){\color{cyan}*11};\node at (12.75,16.85){\color{blue}*37};\node at (18.75,16.85){\color{orange}*2};\node at (20.25,16.85){\color{green}*18};\node at (23.25,16.85){\color{orange}*2};\node at (24.75,16.85){\color{cyan}*11};\node at (26.25,16.85){\color{green}*18};\node at (27.75,16.85){\color{cyan}*11};\node at (29.25,16.85){\color{green}*18};\node at (30.75,16.85){\color{orange}*2};\node at (3.75,15.35){\color{purple}*33};\node at (8.25,15.35){\color{cyan}*11};\node at (11.25,15.35){\color{blue}*37};\node at (15.75,15.35){\color{orange}*3};\node at (18.75,15.35){\color{green}*18};\node at (23.25,15.35){\color{cyan}*11};\node at (24.75,15.35){\color{orange}*3};\node at (26.25,15.35){\color{green}*18};\node at (27.75,15.35){\color{cyan}*11};\node at (29.25,15.35){\color{green}*18};\node at (30.75,15.35){\color{orange}*3};\node at (2.25,13.85){\color{purple}*33};\node at (6.75,13.85){\color{cyan}*12};\node at (9.75,13.85){\color{blue}*38};\node at (15.75,13.85){\color{orange}*3};\node at (17.25,13.85){\color{green}*19};\node at (23.25,13.85){\color{orange}*3};\node at (24.75,13.85){\color{cyan}*12};\node at (26.25,13.85){\color{green}*19};\node at (27.75,13.85){\color{cyan}*12};\node at (29.25,13.85){\color{green}*19};\node at (30.75,13.85){\color{orange}*3};\node at (5.25,12.35){\color{cyan}*12};\node at (8.25,12.35){\color{blue}*38};\node at (12.75,12.35){\color{orange}*4};\node at (15.75,12.35){\color{green}*19};\node at (23.25,12.35){\color{cyan}*12};\node at (24.75,12.35){\color{orange}*4};\node at (26.25,12.35){\color{green}*19};\node at (27.75,12.35){\color{cyan}*12};\node at (29.25,12.35){\color{green}*19};\node at (30.75,12.35){\color{orange}*4};\node at (6.75,10.85){\color{blue}*39};\node at (12.75,10.85){\color{orange}*4};\node at (14.25,10.85){\color{green}*20};\node at (21.75,10.85){\color{cyan}*13};\node at (23.25,10.85){\color{orange}*4};\node at (24.75,10.85){\color{cyan}*13};\node at (26.25,10.85){\color{green}*20};\node at (27.75,10.85){\color{cyan}*13};\node at (29.25,10.85){\color{green}*20};\node at (30.75,10.85){\color{orange}*4};\node at (5.25,9.35){\color{blue}*39};\node at (9.75,9.35){\color{orange}*5};\node at (12.75,9.35){\color{green}*20};\node at (20.25,9.35){\color{cyan}*13};\node at (23.25,9.35){\color{cyan}*13};\node at (24.75,9.35){\color{orange}*5};\node at (26.25,9.35){\color{green}*20};\node at (27.75,9.35){\color{cyan}*13};\node at (29.25,9.35){\color{green}*20};\node at (30.75,9.35){\color{orange}*5};\node at (3.75,7.85){\color{blue}*40};\node at (9.75,7.85){\color{orange}*5};\node at (11.25,7.85){\color{green}*21};\node at (18.75,7.85){\color{cyan}*14};\node at (21.75,7.85){\color{purple}*27};\node at (23.25,7.85){\color{orange}*5};\node at (24.75,7.85){\color{cyan}*14};\node at (26.25,7.85){\color{green}*21};\node at (27.75,7.85){\color{cyan}*14};\node at (29.25,7.85){\color{green}*21};\node at (30.75,7.85){\color{orange}*5};\node at (2.25,6.35){\color{blue}*40};\node at (6.75,6.35){\color{orange}*6};\node at (9.75,6.35){\color{green}*21};\node at (17.25,6.35){\color{cyan}*14};\node at (20.25,6.35){\color{purple}*27};\node at (23.25,6.35){\color{cyan}*14};\node at (24.75,6.35){\color{orange}*6};\node at (26.25,6.35){\color{green}*21};\node at (27.75,6.35){\color{cyan}*13};\node at (29.25,6.35){\color{green}*21};\node at (30.75,6.35){\color{orange}*6};\node at (6.75,4.85){\color{orange}*6};\node at (8.25,4.85){\color{green}*22};\node at (15.75,4.85){\color{cyan}*15};\node at (18.75,4.85){\color{purple}*28};\node at (23.25,4.85){\color{orange}*6};\node at (24.75,4.85){\color{cyan}*15};\node at (26.25,4.85){\color{green}*22};\node at (27.75,4.85){\color{cyan}*15};\node at (29.25,4.85){\color{green}*22};\node at (30.75,4.85){\color{orange}*6};\node at (3.75,3.35){\color{orange}*7};\node at (6.75,3.35){\color{green}*22};\node at (14.25,3.35){\color{cyan}*15};\node at (17.25,3.35){\color{purple}*28};\node at (23.25,3.35){\color{cyan}*15};\node at (24.75,3.35){\color{orange}*7};\node at (26.25,3.35){\color{green}*22};\node at (27.75,3.35){\color{cyan}*15};\node at (29.25,3.35){\color{green}*22};\node at (30.75,3.35){\color{orange}*7};\node at (3.75,1.85){\color{orange}*7};\node at (5.25,1.85){\color{green}*23};\node at (12.75,1.85){\color{cyan}*16};\node at (15.75,1.85){\color{purple}*29};\node at (23.25,1.85){\color{orange}*7};\node at (24.75,1.85){\color{cyan}*16};\node at (26.25,1.85){\color{green}*23};\node at (27.75,1.85){\color{cyan}*16};\node at (29.25,1.85){\color{green}*23};\node at (30.75,1.85){\color{orange}*7};\node at (3.75,0.35){\color{green}*23};\node at (11.25,0.35){\color{cyan}*16};\node at (14.25,0.35){\color{purple}*29};\node at (18.75,0.35){\color{orange}*8};\node at (21.75,0.35){\color{blue}*34};\node at (23.25,0.35){\color{cyan}*16};\node at (24.75,0.35){\color{orange}*8};\node at (26.25,0.35){\color{green}*23};\node at (27.75,0.35){\color{cyan}*16};\node at (29.25,0.35){\color{green}*23};\node at (30.75,0.35){\color{orange}*8};\node at (0.75,-1.15){\color{green}*25};\node at (2.25,-1.15){\color{purple}*33};\node at (3.75,-1.15){\color{purple}*33};\node at (5.25,-1.15){\color{purple}*32};\node at (6.75,-1.15){\color{purple}*32};\node at (8.25,-1.15){\color{purple}*31};\node at (9.75,-1.15){\color{purple}*31};\node at (11.25,-1.15){\color{purple}*30};\node at (12.75,-1.15){\color{purple}*30};\node at (14.25,-1.15){\color{purple}*29};\node at (15.75,-1.15){\color{purple}*29};\node at (17.25,-1.15){\color{purple}*28};\node at (18.75,-1.15){\color{purple}*28};\node at (20.25,-1.15){\color{purple}*27};\node at (21.75,-1.15){\color{purple}*27};\node at (0.75,-2.65){\color{green}*24};\node at (2.25,-2.65){\color{blue}*40};\node at (3.75,-2.65){\color{blue}*40};\node at (5.25,-2.65){\color{blue}*39};\node at (6.75,-2.65){\color{blue}*39};\node at (8.25,-2.65){\color{blue}*38};\node at (9.75,-2.65){\color{blue}*38};\node at (11.25,-2.65){\color{blue}*37};\node at (12.75,-2.65){\color{blue}*37};\node at (14.25,-2.65){\color{blue}*36};\node at (15.75,-2.65){\color{blue}*36};\node at (17.25,-2.65){\color{blue}*35};\node at (18.75,-2.65){\color{blue}*35};\node at (20.25,-2.65){\color{blue}*34};\node at (21.75,-2.65){\color{blue}*34};\node at (0.75,-4.15){\color{green}*24};\node at (2.25,-4.15){\color{blue}*40};\node at (3.75,-4.15){\color{blue}*40};\node at (5.25,-4.15){\color{blue}*39};\node at (6.75,-4.15){\color{blue}*39};\node at (8.25,-4.15){\color{blue}*38};\node at (9.75,-4.15){\color{blue}*38};\node at (11.25,-4.15){\color{blue}*37};\node at (12.75,-4.15){\color{blue}*37};\node at (14.25,-4.15){\color{blue}*36};\node at (15.75,-4.15){\color{blue}*36};\node at (17.25,-4.15){\color{blue}*35};\node at (18.75,-4.15){\color{blue}*35};\node at (20.25,-4.15){\color{blue}*34};\node at (21.75,-4.15){\color{blue}*34};\node at (0.75,-5.65){\color{green}*25};\node at (2.25,-5.65){\color{purple}*33};\node at (3.75,-5.65){\color{purple}*33};\node at (5.25,-5.65){\color{purple}*32};\node at (6.75,-5.65){\color{purple}*32};\node at (8.25,-5.65){\color{purple}*31};\node at (9.75,-5.65){\color{purple}*31};\node at (11.25,-5.65){\color{purple}*30};\node at (12.75,-5.65){\color{purple}*30};\node at (14.25,-5.65){\color{purple}*29};\node at (15.75,-5.65){\color{purple}*29};\node at (17.25,-5.65){\color{purple}*28};\node at (18.75,-5.65){\color{purple}*28};\node at (20.25,-5.65){\color{purple}*27};\node at (21.75,-5.65){\color{purple}*27};
			
			\node at (26.75,-0.5){\color{orange}B1 even cells: *No.};
			\node at (26.75,-1.5){\color{cyan}B1 odd cells: *No.};
			\node at (26.75,-2.5){\color{yellow}B1 s-cells: *No.};
			\node at (26.75,-3.5){\color{green}C2: *No.};
			\node at (26.75,-4.5){\color{purple}D2 part1: *No.};
			\node at (26.75,-5.5){\color{blue}D2 part2: *No.};
			\end{scope}  
			
			\end{tikzpicture}
			\caption{A latin representation of $K_{15,17,21}$, which is obtained from our method, and the trades in it.}
			\label{fig:finally}
		\end{figure}
		
		
		\bibliographystyle{plain}
		
	}
\end{document}